\documentclass[preprint,1p]{elsarticle}
\usepackage{amsfonts,amsmath,amsthm,dsfont}

\newtheorem{tw}{Theorem}
\newtheorem{defin}[tw]{Definition}
\newtheorem{ass}{Assumption}
\newtheorem{lem}[tw]{Lemma}
\newtheorem{cor}[tw]{Corollary}
\newtheorem{prop}[tw]{Proposition}
\newdefinition{rem}[tw]{Remark}
\DeclareMathOperator*{\esup}{ess\,sup}
\DeclareMathOperator*{\tr}{tr}

\begin{document}
\def\r0{\mathbb{R}^d_0}
    \def\E{\mathbb{E}}
    \def\GE{\hat{\mathbb{E}}}
    \def\cadlag{c\`{a}dl\`{a}g }
        \def\cliprd{C_{b,Lip}(\mathbb{R}^d)}
    \def\cliprn{C_{b,Lip}(\mathbb{R}^n)}
    \def\cliprdn{C_{b,Lip}(\mathbb{R}^{d\times n})}
    \def\cliprm{C_{b,Lip}(\mathbb{R}^m)}
    \def\cp{C^{\infty}_{p}(\mathbb{R}^n)}
    \def\clipr{C_{b,Lip}(\mathbb{R})}
    \def\lipt{Lip(\Omega_t)}
        \def\lipT{Lip(\Omega_T)}
    \def\lip{Lip(\Omega)}
    \def\qB{\langle B\rangle}
    \def\P{\mathbb{P}}
    \def\D{\mathbb{D}}
    \def\I{\mathds{1}}
    \def\N{\mathbb{N}}
        \def\n{\mathcal{N}}
    \def\L12{\mathbb{L}^{1,2}}
    \def\R{\mathbb{R}}
    \def\ae{\mathcal{A}_{\EE}}
    \def\Z{\mathbb{Z}}
        \def\A{\mathcal{A}}
    \def\H{\mathcal{H}}
     \def\G{\mathcal{G}}
          \def\a{\mathcal{A}}
    \def\C{\mathbb{C}}
        \def\L{\mathbb{L}}
    \def\Q{\mathbb{Q}}
        \def\q{\mathcal{Q}}
    \def\S{\mathcal{S}}
        \def\s{\mathbb{S}}
    \def\B{\mathcal{B}}
        \def\v{\mathcal{V}}
	\def\u{\mathcal{U}}
    \def\b{\mathbb{B}}
    \def\fil{\mathbb{F}}
    \def\F{\mathcal{F}}
    \def\EE{\mathcal{E}}
    \def\m{\mathcal{M}}
    \def\p{\mathcal{P}}
    \def\ito{It\^o }
    \def\levy{L\'evy }
    \def\itolevy{It\^o-L\'evy }
        \def\levykhintchine{L\'evy-Khintchine }
    \def\LG{L^2_G(\Omega)}
    \def\LGp{L^p_G(\Omega)}
    \def\LGT{L^2_G(\Omega_T)}
        \def\LGpl{L^2_{G_{\epsilon}}(\mathcal{F}_T)}
    \def\lgT{L^2_G(0,T)}
     \def\mgT{M^2_G(0,T)}
    \def\hgTR{\mathcal{H}^2_G([0,T]\times\mathbb{R}^d_0)}
	\def\hhR{\hat{\mathcal{H}}^2_G([0,T]\times\mathbb{R}^d_0)}
    \def\hgT{\mathcal{H}^2_G(0,T)}
    \def\mg1T{M^1_G(0,T)}
	\def\Mg1T{\mathcal{M}^1_G(0,T)}
    \def\hMT{\hat{\mathcal{M}}^2_G(0,T)}
    \def\hM1T{\hat{\mathcal{M}}^1_G(0,T)}
    \def\hMTp{\hat{\mathcal{M}}^p_G(0,T)}
    \def\mgT{{M}^1_G(0,T)}
    \def\d0{\mathbb{D}_0(\R^+,\R^{d})}
    \def\da{\mathbb{D}_0(\R^+,\R^{2d})}

\begin{frontmatter}
\author{Krzysztof Paczka \fnref{fn1}}
\ead{k.j.paczka@cma.uio.no}
\address{Centre of Mathematics for Applications, University of Oslo \\P.O. Box 1053 Blindern, 0316, Oslo, Norway}
\fntext[fn1]{The research leading to these results has received funding from the European Research Council under the European Community's Seventh Framework Programme (FP7/2007-2013) / ERC grant agreement no [228087].}

\title{It\^o calculus and jump diffusions for $G$-L\'evy processes}

\begin{abstract}
	The paper considers the integration theory for $G$-\levy processes with finite activity. We introduce the \itolevy integrals, give the \ito formula for them and establish SDE's, BSDE's and decoupled FBSDE's driven by $G$-\levy processes. In order to develop such a theory, we prove two key results: the representation of the sublinear expectation associated with a $G$-\levy process and a characterization of random variables in $L^p_G(\Omega)$ in terms of their quasi-continuity. 
\end{abstract}

\end{frontmatter}

\section{Introduction}
In recent years much effort has been made to develop the theory of sublinear expectations connected with the volatility uncertainty and so-called $G$-Brownian motion. $G$-Brownian motion was introduced by Shige Peng in \cite{Peng_GBM} as a way to incorporate the unknown volatility into financial models. Its theory is tightly associated with the uncertainty problems involving an undominated family of probability measures. Soon other connections have been discovered, not only in the field of financial mathematics, but also in the theory of path-dependent PDE's or 2BSDE's. Thus $G$-Brownian motion and connected $G$-expectation are  very attractive mathematical objects.

Returning however to the original problem of volatility uncertainty in the financial models, one feels that $G$-Brownian motion is not sufficient to model the financial world, as both $G$- and the standard Brownian motion share the same property, which makes them often unsuitable for modelling, namely the continuity of paths. Therefore, it is not surprising that Hu and Peng introduced the process with jumps, which they called $G$-\levy process (see \cite{Peng_levy}). Unfortunately, the theory of $G$-\levy processes is still very undeveloped, especially compared $G$-Brownian motion. The follow-up has been limited to the paper by Ren (\cite{Ren}), which introduces the representation of the sublinear expectation as an upper-expectation.

In this paper we concentrate on establishing the integration theory for $G$-\levy processes with finite activity. We introduce the integral w.r.t. the jump measure associated with the pure jump $G$-\levy process (Section \ref{sec_integral}), give the \ito formula for general $G$-\itolevy processes (Section \ref{sec_ito_form}) and we look at different typed of differential equations: both forward and backward (Section \ref{sec_diffusions}). The crucial piece of theory needed to obtain those results is the representation of the sublinear expectation as a supremum of ordinary expectations, given in Section \ref{sec_representation}.  The representation theorem, though inspired by the already quoted paper by Ren, also differs substantially, as we give the  characterization of the probability measures used in the representation as the law of some stochastic integral. Another result worth mentioning, is the complete characterization of the space $L^1_G(\Omega)$  in terms of (quasi)-continuity  (Section \ref{sec_quasi_anal}).

\section{Preliminaries}\label{sec_preliminaries}
	Let $\Omega$ be a given set and $\H$  be a vector lattice of real functions defined on $\Omega$, i.e. a linear space containing $1$ such that $X\in\H$ implies $|X|\in\H$. We will treat elements of $\H$ as random variables.
	\begin{defin}\label{def_sublinear_exp}
		\emph{A sublinear expectation} $\E$ is a functional $\E\colon \H\to\R$ satisfying the following properties
		\begin{enumerate}
			\item \textbf{Monotonicity:} If $X,Y\in\H$ and $X\geq Y$ then $\E[ X]\geq\E [Y]$.
			\item \textbf{Constant preserving:} For all $c\in\R$ we have $\E [c]=c$.
			\item \textbf{Sub-additivity:} For all $X,Y\in\H$ we have $\E [X] - \E[Y]\leq\E [X-Y]$.
			\item \textbf{Positive homogeneity:} For all $X\in\H$  we have $\E [\lambda X]=\lambda\E [X]$, $\forall\,\lambda\geq0$.
		\end{enumerate}
		The triple $(\Omega,\H,\E)$ is called \emph{a sublinear expectation space}.
	\end{defin}
	
	We will consider a space $\H$ of random variables having the following property: if\break $X_i\in\H,\ i=1,\ldots n$ then
	\[
		\phi(X_1,\ldots,X_n)\in\H,\quad \forall\ {\phi\in\cliprn},
	\]
	where $\cliprn$ is the space of all bounded Lipschitz continuous functions on $\R^n$. We will express the notions of a distribution and an independence of the random vectors using test functions in $\cliprn$.
	\begin{defin}
		For an $n$-dimensional random vector $X=(X_1,\ldots,X_n)$ define the functional ${\fil}_X$ on $\cliprn$ as
		\[
			{\fil}_X[\phi]:=\E[\phi(X)],\quad \phi\in\cliprn.
		\]
		We will call the functional ${\fil}_X$ the distribution of $X$.  We say that two $n$-dimensional random vectors $X_1$ and $X_2$ (defined possibly on different sublinear expectation spaces) are identically distributed, if their distributions $\fil_{X_1}$ and $\fil_{X_2}$ are equal.
		
		An $m$-dimensional random vector $Y=(Y_1,\ldots,Y_m)$ is said to be independent of an $n$-dimensional random vector $X=(X_1,\ldots,X_n)$ if 
		\[
			\E[\phi(X,Y)]=\E[\E[\phi(x,Y)]_{x=X}]. \quad  \forall \phi\in C_{b,Lip}(\R^n\times \R^m).
		\]
		\end{defin}
	Now we give the definition of $G$-\levy process (after \cite{Peng_levy}).	
	\begin{defin}	
		Let $X=(X_t)_{t\geq0}$ be a $d$-dimensional \cadlag process on a sublinear expectation space $(\Omega,\H, \E)$. We say that $X$ is a \levy process if:
		\begin{enumerate}
			\item $X_0=0$,
			\item for each $t,s\geq 0$ the increment $X_{t+s}-X_{t}$ is independent of $(X_{t_1},\ldots, X_{t_n})$ for every $n\in\N$ and every partition $0\leq t_1\leq\ldots\leq t_n\leq t$,
			\item the distribution of the increment $X_{t+s}-X_t,\ t,s\geq 0$ is stationary, i.e.\ does not depend on $t$.
		\end{enumerate}
		Moreover, we say that a \levy process $X$ is a $G$-\levy process, if satisfies additionally following conditions
		\begin{enumerate}\setcounter{enumi}{3}
			\item there a $2d$-dimensional \levy process $(X^c_t,X^d_t)_{t\geq0}$ such for each $t\geq0$ $X_t=X_t^c+X_t^d$,
			\item\label{condition} processes $X^c$ and $X^d$ satisfy the following growth conditions
			\[
					\lim_{t\downarrow0} \E[|X^c_t|^3]t^{-1}=0;\quad \E[|X^d_t|]<Ct\ \textrm{for all}\ t\geq0.
			\]
		\end{enumerate}
	\end{defin}
	\begin{rem}
		The condition \ref{condition} implies that $X^c$ is a $d$-dimensional generalized $G$-Brownian motion (in particular, it has continuous paths), whereas the jump part $X^d$ is of finite variation.
	\end{rem}
	Peng and Hu noticed in their paper that each $G$-\levy process $X$ might be characterized by a non-local operator $G_X$.
	\begin{tw}[\levykhintchine representation, Theorem 35 in \cite{Peng_levy}]
		Let $X$ be a $G$-\levy process in $\R^d$. For every $f\in C^3_b(\R^d)$ such that $f(0)=0$ we put
		\[
			G_X[f(.)]:=\lim_{\delta\downarrow 0}\, \E[f(X_{\delta})]\delta^{-1}.
		\]
		The above limit exists. Moreover, $G_X$ has the following \levykhintchine representation
		\[
			G_X[f(.)]=\sup_{(v,p,Q)\in \u}\left\{\int_{\r0\setminus\{0\}} f(z)v(dz)+\langle Df(0),p\rangle + \frac{1}{2}\tr[D^2f(0)QQ^T] \right\},
		\]
		where $\r0:=\R^d\setminus\{0\}$, $\u$ is a subset $\u\subset \m(\r0)\times \R^d\times\R^{d\times d}$ and $\m(\r0)$ is a set of all Borel measures on $(\r0,\B(\r0))$. We know additionally that $\u$ has the property
		\begin{equation}\label{eq_property_of_u}
			\sup_{(v,p,Q)\in \u}\left\{\int_{\r0} |z|v(dz)+|p|+ \tr[QQ^T] \right\}<\infty.		
		\end{equation}

	\end{tw}
	\begin{tw}[Theorem 36 in \cite{Peng_levy}]
			Let $X$ be a $d$-dimensional $G$-\levy process.  For each $\phi\in\cliprd$, define $u(t,x):=\E[\phi(x+X_t)]$. Then $u$ is the unique viscosity solution of the following integro-PDE
			\begin{align}\label{eq_integroPDE}
				0=&\partial_tu(t,x)-G_X[u(t,x+.)-u(t,x)]\notag\\
				=&\partial_tu(t,x)-\sup_{(v,p,Q)\in \u}\left\{\int_{\r0} [u(t,x+z)-u(t,x)]v(dz)\right.\notag\\
				&\left.+\langle Du(t,x),p\rangle + \frac{1}{2}\tr[D^2u(t,x)QQ^T] \right\}
			\end{align}
			with initial condition $u(0,x)=\phi(x)$.
	\end{tw}
	
	It turns out that the set $\u$ used to represent the non-local operator $G_X$ fully characterize $X$, namely having $X$ we can define $\u$ satysfying eq.\ (\ref{eq_property_of_u}) and vice versa.
	\begin{tw}
		Let $\u$ satisfy (\ref{eq_property_of_u}). Consider the canonical space $\Omega:=\d0$ of all \cadlag functions taking values in $\R^{d}$ equipped with the Skorohod topology. Then there exists a sublinear expectation $\GE$ on $\d0$ such that the canonical process $(X_t)_{t\geq0}$ is a $G$-\levy process satisfying \levykhintchine representation with the same set $\u$.
	\end{tw}
	The proof might be found in \cite{Peng_levy} (Theorem 38 and 40). We will give however the construction of $\GE$, as it is important to understand it.
	
		Begin with defining the sets of random variables. We denote $\Omega_T:=\{\omega_{.\wedge T}\colon \omega\in\Omega\}$. Put
		\begin{align*}
			\lipT:=&\{\xi\in L^0(\Omega)\colon \xi=\phi(X_{t_1},X_{t_2}-X_{t_1},\ldots,X_{t_n}-X_{t_{n-1}}),\\ &\phi\in C_{b,Lip}(\R^{d\times n}),\ 0\leq t_1<\ldots<t_n<T\},
		\end{align*}
		where $X_t(\omega)=\omega_t$ is the canonical process on the space $\d0$ and $L^0(\Omega)$ is the space of all random variables, which are measurable to the filtration generated by the canonical process. We also set
		\[
			\lip:=\bigcup_{T=1}^{\infty}\, \lipT.
		\]
		Firstly, consider the random variable $\xi=\phi(X_{t+s}-X_{t})$, $\phi\in\cliprd$. We define
		\[
			\GE[\xi]:=u(s,0),
		\]
		where $u$ is a unique viscosity solution of integro-PDE (\ref{eq_integroPDE}) with the initial condition $u(0,x)=\phi(x)$. For general
		\[
			\xi=\phi(X_{t_1},X_{t_2}-X_{t_1},\ldots,X_{t_n}-X_{t_{n-1}}),\quad \phi\in C_{b,Lip}(\R^{d\times n})
		\]
		we set $\GE[\xi]:=\phi_n$, where $\phi_n$ is obtained via the following iterated procedure
		\begin{align*}
			\phi_1(x_1,\ldots,x_{n-1})&=\GE[\phi(x_1,\ldots, X_{t_n}-X_{t_{n-1}})],\\
			\phi_2(x_1,\ldots,x_{n-2})&=\GE[\phi_1(x_1,\ldots, X_{t_{n-1}}-X_{t_{n-2}})],\\
			&\vdots\\
			\phi_{n-1}(x_1)&=\GE[\phi_{n-1}(x_1,X_{t_{2}}-X_{t_{1}})],\\
			\phi_n&=\GE[\phi_{n-1}(X_{t_{1}})].
		\end{align*}
		Lastly, we extend definition of $\GE[.]$ on the completion of $\lipT$ (respectively $\lip$) under the norm $\|.\|_p:=\GE[|.|^p]^{{1}/{p}},\ p\geq1$. We denote such a completion by $L^p_G(\Omega_T)$\break (or resp. $L^p_G(\Omega)$).
		
		It is also important to note that using this procedure we may in fact define the conditional sublinear expectation $\GE[\xi|\Omega_t]$. Namely, w.l.o.g. we may assume that $t=t_i$ for some $i$ and then
		\[
			\GE[\xi|\Omega_{t_i}]:=\phi_{n-i}(X_{t_{0}},X_{t_1}-X_{t_0},\ldots,X_{t_i}-X_{t_{i-1}}).
		\]
		One can easily prove that such an operator is continuous w.r.t. the norm $\|.\|_1$ and might be extended to the whole space $L^1_G(\Omega)$. By construction above, it is clear that the conditional expectation satisfies the tower property, i.e. is dynamically consistent.

\section{Representation of $\GE[.]$ as an upper-expectation}\label{sec_representation}
In the rest of this paper we will work on the canonical space $\Omega:=\d0$ and the sublinear expectation $\GE[.]$ such that the canonical process $X$ is a $G$-\levy process satisfying the \levykhintchine representation for some set $\u$. Just as in the case of $G$-Brownian motion it is reasonable to ask, if we can represent $\GE[.]$ as an upper-expectation (i.e. supremum of expectations related to the probability measures on $\d0$). Moreover, can we describe these probability measures as laws of some processes on $\d0$?

These questions have been partially addressed in the paper by Liying Ren \cite{Ren}. He showed that for every $\u$ there exist a relatively compact family of probability measures $\mathfrak{P}$ such that
\[
	\GE[\xi]=\sup_{\P\in\mathfrak{P}}\, \E^{\P}[\xi],\quad \xi\in \lip.
\]

The aim of this section is to characterize these probability measures as a law of some stochastic integral. As in the original paper \cite{Denis_function_spaces}, which represents a $G$-expectation as an upper-expectation, we will use the dynamic programming method and we will prove first the dynamic programming principle (DPP) in our set-up. From this point we will work under the following assumption.

\begin{ass}\label{ass1}
Let a canonical process $X$ be a $G$-\levy process in $\R^d$ on a sublinear expectation space $(\d0, L^1_G(\Omega), \GE)$. Let $\u\subset \m(\r0)\times \R^d\times\R^{d\times d}$ be a set used in the \levykhintchine representation of $X$ \eqref{eq_integroPDE} satisfying \eqref{eq_property_of_u}. Let 
\[
	\v:=\{v\in  \m(\r0)\colon \exists (p,q)\in \R^d\times\R^{d\times d}\textrm{ such that }(v,p,q)\in \u\}
\]
and let $\G_{\B}$ denote the set of all Borel function $g\colon\R^d\to \R^d$ such that $g(0)=0$.

We will assume that there exists a measure $\mu\in \m(\R^d)$ such that
\[
	\int_{\r0}|z|\mu(dz)<\infty\quad \textrm{and}\quad \mu(\{0\})=0
\]
and for all $v\in\v$ there exists a function $g_v\in\G_{\B}$ satisfying the following condition
\[
	v(B)=\mu(g_v^{-1}(B))\quad \forall B\in\B(\r0).
\]
Last assume that there exists $0<q<1$ such that
\[
	\sup_{v\in\v}\int_{0<|z|<1}|z|^qv(dz)<\infty.
\]
\end{ass}

\begin{rem}
	For any measure $v\in\v$ consider
	\[\G_v:=\{g\in\G_{\B}  \colon v(B)=\mu(g_v^{-1}(B))\quad \forall B\in\B(\r0)\}\]
	 Under Assumption \ref{ass1} we know that $\G_{v}$ contains at least one element, but in general the cardinality of the set might be up to $2^{\aleph_0}$. Since we want to have one-to-one relation between $v$ and $g_v\in\G_v$ we will understand the latter as a unique representative element of the set $\G_v$ chosen using the axiom of choice.
	 
	Now may consider a different parametrizing set in the \levykhintchine formula. Namely using
	\[
		\tilde \u:=\{(g_v,p,q)\in \G_{\B}\times \R^d\times\R^{d\times d}\colon (v,p,q)\in\u\}.
	\]
	it is elementary that the  equation \eqref{eq_integroPDE} is equivalent to the following equation
				\begin{align}\label{eq_integroPDE2}
				0=&\partial_tu(t,x)-\sup_{(g,p,Q)\in \tilde\u}\left\{\int_{\r0} [u(t,x+g(z))-u(t,x)]\mu(dz)\right.\notag\\
				&\left.+\langle Du(t,x),p\rangle + \frac{1}{2}\tr[D^2u(t,x)QQ^T] \right\}.
			\end{align}
\end{rem}

\begin{rem}
We will call a $G$-\levy process $X$ \emph{a $G$-\levy process with finite activity}, if $\lambda:=\sup_{v\in\v}v(\r0)<\infty$. If moreover $d=1$, Assumption 1 is always satisfied, as we can take the Lebesgue measure on the interval $[0,\lambda]$ as $\mu$ and $g_v:=F^{-1}_v$, where $F^{-1}_v$ is a general inverse of the cumulative distribution function $F_v$ of the measure $v$. For $d>1$ we can use the Knothe-Rosenblatt rearrangement to transport measure $\mu$ to measure $v$ (for details see \cite{Villiani}, p. 8-9).
\end{rem}

We are ready now to introduce some stochastic control problem associated with IPDE \eqref{eq_integroPDE}. Let $(\tilde \Omega,\F,\P_0)$ be a probability space carrying a Brownian motion $W$ and a \levy process with a \levy triplet $(0,0,\mu)$, which is independent of $W$. Let $N(dt,dz)$ be a Poisson random measure associated with that \levy process. Define
$N_t=\int_{\r0}xN(t,dx)$, which is finite $\P_0$-a.s. as we assume that $\mu$ integrates $|x|$. Moreover, in the finite activity case $\lambda=\sup_{v\in\v}v(\r0)<\infty$ we define the Poisson process $M$ with intensity $\lambda$ by putting
$M_t=N(t,\r0)$.

We define the filtration generated by $W$ and $N$:
\begin{align*}
	\F_t:=&\sigma\{W_s,\ N_s\colon 0\leq s\leq t\}\vee\tilde\n;\ \tilde\n:=\{A\in\F\colon \P_0(A)=0\};\;\ \fil:=(\F_t)_{t\geq0};\\
	\F^s_t:=&\sigma\{W_{u}-W_s,\ N_u-N_s\colon s\leq u\leq t\}\vee\tilde\n\quad 0\leq s\leq t;\ \fil^s:=(\F^s_t)_{t\geq s},\  s\geq 0.
\end{align*}

\begin{defin}\label{def_theta}
	Introduce a set of integrands $\a_{t,T}^{\u}$, $0 \leq t<T$, associated with $\u$ as a set of all processes $\theta=(\theta^d,\theta^{1,c}, \theta^{2,c})$ defined on $]t,T]$ satisfying the following properties
	\begin{enumerate}
		\item $(\theta^{1,c},\theta^{2,c})$ is $\fil$-adapted process and $\theta^d$ is $\fil$-predictable random field on $]t,T]\times \R^d$.
		\item For $\P_0$-a.a. $\omega\in\tilde\Omega$ and a.e. $s\in]t,T]$ we have that 
		\[
			(\theta^d(s,.)(\omega),\theta^{1,c}_s(\omega), \theta^{2,c}_s(\omega))\in \tilde\u.
		\]
		\item $\theta$ satisfies the following integrability condition
		\[
		\E^{\P_0}\left[\int_t^T\left[|\theta^{1,c}_s|+|\theta^{2,c}_s|^2+\int_{\r0}|\theta^{d}(s,z)|\mu(dz)\right]ds\right]<\infty.
		\]
	\end{enumerate}
\end{defin}
		We stress that in point 1. by the predictable $\sigma$-algebra we mean a $\sigma$-algebra on $[0,T]\times\R^d\times \tilde\Omega$ as defined for example in \cite{applebaum}, Section 4.1, p. 216. Note that thanks to eq. \eqref{eq_property_of_u} we have that $\int_{\r0}|\theta^{d}(s,z)|\mu(dz)<\infty$ $\P_0$-a.s. for a.e. $s$ if $\theta^d$ satisfies condition 2, thus condition 3 in the definition above has sense. 
	
	For $\theta\in \a_{0,\infty}^{\u}$ denote the following \levy -\ito integral as
	\[
		B^{t,\theta}_T=\int_{t}^T\, \theta^{1,c}_sds+\int_{t}^T\, \theta^{2,c}_sdW_s+ \int_{]t,T]}\int_{\r0}\, \theta^{d}(s,z)N(ds,dz).
	\]
	The first integral is taken in the Lebesgue sense, the second: in the \ito sense, whereas the last one is defined pathwise. All integrals have sense thanks to condition 3 in the definition of $\a^{\u}_{t,T}$. In particular, to see that the pathwise integral is defined properly and has desired properties see Chapter II, section 1 in \cite{Jacod}.
	
	Lastly, for a fixed $\phi\in \cliprd$ and fixed $T>0$ define for each $(t,x)\in[0,T]\times\R^d$
	\begin{align*}
		u(t,x)&:=\sup_{\theta\in\a_{t,T}^{\u}}\E^{\P_0} [\phi(x+B^{t,\theta}_T)].
	\end{align*}

The most crucial results of this section are to be found in these three theorems.
\begin{tw}[DPP]\label{tw_DPP}
	We have that for every $h>0$ such that	$t+h<T$ one has
	\[
		u(t,x)=\sup_{\theta\in \a_{t,t+h}^{\u}}\E^{\P_0}[u(t+h,x+B^{t,\theta}_{t+h})].
	\]
\end{tw}
\begin{tw}\label{tw_viscosity_sol_iPDE}
	$u$ is the viscosity solution of the following integro-PDE
	\[
		\partial_tu(t,x)+G_X[u(t,x+.)-u(t,x)]=0
	\]
	with the terminal condition $u(T,x)=\phi(x)$.
\end{tw}
\begin{tw}\label{tw_representation_for_lip}
	Let $\xi\in \lipT$ has the representation
		\[
			\xi=\phi(X_{t_1},X_{t_2}-X_{t_1},\ldots,X_{t_n}-X_{t_{n-1}}),\quad \phi\in C_{b,Lip}(\R^{d\times n}).
		\]
	Then
	\begin{align*}
		\GE[\xi]=&\sup_{\theta\in \a_{0,\infty}^{\u}}\, \E^{\P_0}[\phi(B^{0,\theta}_{t_1},B^{t_1,\theta}_{t_2},\ldots,B^{t_{n-1},\theta}_{t_n})].
	\end{align*}
\end{tw}

Lastly, we will give the following easy corollary to Theorem \ref{tw_representation_for_lip}.
\begin{cor}\label{col_representation_L_G}
	Let $\xi\in L^1_G(\Omega)$. Then one has
	\[
		\GE[\xi]=\sup_{\theta\in \a^{\u}_{0,\infty}}\, \E^{\P_0}[\xi\circ B.^{0,\theta}],
	\]
	where $\circ$ denotes composition of functions. We treat here the \ito integral $ B{.}^{0,\theta}$ as a function of $\Omega$ taking values in \cadlag functions. We can also write it in a different form, defining $\P^{\theta}:=\P_0\circ (B{.}^{0,\theta})^{-1}$. Then
		\[
		\GE[\xi]=\sup_{\theta\in \a^{\u}_{0,\infty}}\, \E^{\P^{\theta}}[\xi].
	\]
\end{cor}

We will leave the proofs of Theorems \ref{tw_DPP}, \ref{tw_viscosity_sol_iPDE} and \ref{tw_representation_for_lip} to the Appendix, as they are very similar to proofs already published, which can be found \cite{Denis_function_spaces}. 

\section{Capacity and related topics}\label{sec_quasi_anal}
In the remaining part of the paper we will assume that the $G$-\levy process $X$ has finite activity.
\begin{ass}
	We assume from this point to the end of the paper that $X$ is a $G$-\levy process with finite activity, i.e.
	\[
		\lambda=\sup_{v\in\v}\ v(\r0)<\infty.
	\]
	Without loss of generality we will also assume that $\lambda=1$ and that also $ \mu(\r0)=1$.
\end{ass}
This assumption will allow us to define the \ito integral w.r.t. the Poisson random measure associated with $X$. But before we can do that, we need to establish some properties of the paths of a $G$-\levy process with finite activity and relate different notions of regularity of random variables. In both cases, it is natural to consider the capacity framework.

\subsection{Quasi-sure properties of the paths}
\begin{defin}
	For the sublinear expectation $\E[.]$ with the representation
	\[
		\E [.]=\sup_{\Q \in \mathfrak{P}}\, \E^{\Q}[.]
	\]
	we introduce the capacity related to $\E[.]$ as
	\[
		c(A)=\sup_{\Q \in \mathfrak{P}} \Q(A), \quad A\in \B(\Omega).
	\]	
	We say that the set $A$ is polar, if $c(A)=0$. We say that the property holds quasi-surely (q.s.), if it holds outside the polar set.
\end{defin}

Note that due to the representation of sublinear expectation in Corollary \ref{col_representation_L_G}, for each $G$-\levy process with finite activity on the sublinear expectation space $(\Omega,L^1_G(\Omega),\GE)$ we can associate the family of probabilities $\mathfrak{P}=\{\P^{\theta}\colon \theta\in\a^{\u}_{0,\infty}\}$ and thus associate the capacity $c$, too. From now on, whenever we mention the property holding quasi-surely, it will be related to this particular capacity.

We will prove the following proposition, which enables us to work on the paths of a $G$-\levy process.
\begin{prop}\label{prop_finite_activity}
	Let $X$ be a canonical process on the canonical sublinear expectation space $(\Omega,L^1_G(\Omega), \GE)$, such that $X$ is a $G$-\levy process with finite activity under $\GE$. Then for each finite interval $[s,t]$ $X$ has finite number of jumps q.s. Hence the use of the term "finite activity" is justified.
\end{prop}

\begin{proof}
 	Fix $0\leq s<t<\infty$. Define the set
 	\[
 		A:=\{\omega\in\Omega\colon t\mapsto X_t(\omega) \textrm{ has infinite number of jumps on the interval }[s,t]\}.
 	\]
 	We will prove that $\P^{\theta}(A)=0$ for each $\theta\in \a^{\u}_{0,\infty}$. Note that $\P^{\theta}=\P_0\circ (B.^{0,\theta})^{-1}$, thus the canonical process under $\P^{\theta}$ has the same law as the integral $t\mapsto B_t^{0,\theta}$ under $\P_0$. Hence
 	\[
 		\P^{\theta}(A)=\P_0(A^{\theta}),
 	\]
 	where \[A^{\theta}:=\{\omega\in\Omega\colon t\mapsto B_t^{0,\theta}(\omega) \textrm{ has infinite number of jumps on the interval }[s,t]\}.\]
 	But we know that the integral has finite activity $\P_0$-a.s. as its jump part is an integral w.r.t. Poisson random measure $N$ and thus it shares the jumps with the Poisson process $M$. Hence $\P_0(A^{\theta})=0$.
\end{proof}
\begin{rem}
	Note that of course the set $A$ is non-empty, as there are plenty of \cadlag functions (i.e. trajectories of $X$), which do not exhibit finite activity. They are however negligible, as those $\omega$'s belong to a polar set.
\end{rem}

\subsection{Spaces of random variables and the relations between them}
Now let us note that the we can extend our sublinear expectation to all random variables $Y$ on $\Omega_T$ (or $\Omega$) for which the following expression has sense
\[
	\GE[Y]:=\sup_{\theta\in A^{\u}_{0,\infty}}\, \E^{\P^{\theta}}[Y].
\]
We can thus can also extend the definition of the norm $\|.\|_p$ and  define following spaces
\begin{enumerate}
	\item Let $L^0(\Omega_T)$ (resp. $L^0(\Omega)$) be the space of all random variables on $\Omega_T$ (resp. $\Omega$). $L^p:=\{X\in L^(\Omega)\colon \|X\|_p<\infty\}$.
		\item $\n:=\{X\in L^0(\Omega)\colon X=0\ q.s.\}$.
	\item $\L^p:=L^p/\n$. $\L^p$ is a Banach space under the norm $\|.\|_p$. As usual we do not distinguish between the equivalence classes and their representatives.
	\item Let $C_b(\Omega_T)$ be the space of all continuous and bounded random variables in $L^0(\Omega_T)$. The completion of $C_b(\Omega_T)$ in the norm $\|.\|_p$ will be denoted as $\L^p_c(\Omega_T)$.
	\item Let $C_{b,lip}(\Omega_T)$ be the space of all Lipschitz continuous random variables in $C_b(\Omega_T)$. The completion of $C_{b,lip}(\Omega_T)$ in the norm $\|.\|_p$ will be denoted as $\L^p_{c,lip}(\Omega_T)$.
\end{enumerate}

We will need the relation between these spaces and $L^p_G(\Omega_T)$ space. In the $G$-Brownian motion case it is well-known that $Lip(\Omega_T)\subset C_b(\Omega_T)$ and $L^p_G(\Omega_T)=\L^p_c(\Omega_T)$. In the case of $G$-\levy process the first inclusion is untrue as the evaluations of the \cadlag paths are not continuous in the Skorohod topology (compare \cite{billingsley}, section 'Finite-Dimensional Sets' in Chapter 3). However Ren was able to prove that $Lip(\Omega_T)\subset \L^1_c(\Omega_T)$ and thus $L^p_G(\Omega_T)\subset \L^p_c(\Omega_T)$ (see Section 4 in \cite{Ren}). This relation is somehow unsatisfactory, because we would like to have a criterion for a random variable to be in our main space $L^p_G(\Omega_T)$, not the opposite. Fortunately we are able to prove that $L^p_G(\Omega_T)=\L^p_c(\Omega_T)$. But first let us prove with the following relation.
\begin{prop}\label{prop_lip_in_Lg}
	We have the following inclusion
	\[C_{b,lip}(\Omega_T)\subset L^1_G(\Omega_T).\]
    As a consequence
    \[\L^p_{c,lip}(\Omega_T)\subset L^p_G(\Omega_T).\]
\end{prop}

Note that we can't directly use the proof from \cite{Denis_function_spaces}, which is based on the Stone-Weierstrass theorem and the tightness of the family $\{\P^{\theta}\colon \theta\in \a^{\u}_{0,T}\}$, because even though $Lip(\Omega_T)$ is an algebra which separates points in $\Omega_T$, but it is not included either in $C_{b,lip}(\Omega_T)$ nor in $C_{b}(\Omega_T)$ as its elements are not continuous. Thus we need to show the proposition directly constructing an appropriate approximative sequence. We will need the following properties.
\begin{lem}\label{lem_modulus}
	For any $\delta>0$ and a \cadlag function $x:[0,T]\to \R^d$ define the following \cadlag modulus
	\[
		\omega_x'(\delta):=\inf_{\pi}\max_{0<i\leq r} \sup_{s,t\in [t_{i-1},t_i[} |x(s)-x(t)|,
	\]
	where infinimum runs over all partitions $\pi=\{t_0,\ldots,t_r\}$ of the interval $[0,T]$ satisfying $0=t_0<t_1<\ldots<t_r=T$ and $t_i-t_{i-1}>\delta$ for all $i=1,2,\ldots,r$.
	Define also
		\begin{align*}
		w_x''(\delta):=&\sup_{\substack{t_1\leq t\leq t_2\\t_2-t_1\leq \delta}}\min\{ |x(s)-x(t_1)|,|x(t_2)-x(s)|\}.
	\end{align*}
	Then
	\begin{enumerate}
		\item 		$w_x''(\delta)\leq w_x'(\delta)$ for all $\delta>0$ and $x\in \d0$.
		\item For every $\epsilon>0$ and a subinterval $[\alpha,\beta[\subset [0,T]$ if  $x$ does not have any jumps of magnitude $>\epsilon$ in the interval $[\alpha,\beta[$ then
		\[
			\sup_{\substack{t_1,t_2\in[\alpha,\beta[\\|t_2-t_1|\leq\delta}} |x(t_1)-x(t_2)|\leq 2w''_x(\delta)+\epsilon.
		\]
		In particular, if $x$ is continuous in $[\alpha,\beta[$, we have the estimate
		\[
			\sup_{\substack{t_1,t_2\in[\alpha,\beta[\\|t_2-t_1|\leq\delta}} |x(t_1)-x(t_2)|\leq 2w''_x(\delta)\leq 2w'_x(\delta).
		\]
		\item The function $x\mapsto w_x'(\delta)$ is upper semicontinuous for all $\delta>0$.
		\item  $\lim_{\delta\downarrow 0}\,w_x'(\delta)\downarrow 0$ for all $x\in \d0$.
	\end{enumerate}
\end{lem}
These properties are standard and might be found in \cite{billingsley} for properties 1, 3 and 4 (see Chapter 3, Lemma 1, eq. (14.39) and (14.46)) and \cite{Parthasarathy} for property 2 (see Lemma 6.4 in Chapter VII).

\begin{proof}[Proof of Proposition \ref{prop_lip_in_Lg}]
	Fix a random variable $Y\in C_{b,lip}(\Omega_T)$. For any $n\in\N$ define the operator $T^n\colon\d0\to \d0$ as
	\[
		T^n(\omega)(t):=\left\{ \begin{array}{ll}
\omega_{\frac{kT}{n}} & \textrm{if }t\in[\frac{kT}{n},\frac{(k+1)T}{n}[,\ k=0,1,\ldots,n-1.\\
\omega_T & \textrm{if }t=T.
\end{array} \right.
	\]
	Define $Y^n:=Y\circ T^n$. Then $Y^n$ depend only on $\{\omega_{{kT}/{n}}\}_{k=0}^{n}$ thus there exists a function $\phi^n\colon \R^{(n+1)\times d}\to\R$ such that
	\[
		Y^n(\omega)=\phi^n(\omega_0,\omega_{\frac{T}{n}},\ldots,\omega_{T}).
	\]
	By the boundedness and Lipschitz continuity of $Y$ we can easily prove that also $\phi^n$ must be bounded and Lipschitz continuous (all we have to do is to consider the paths, which are constant on the intervals $[kT/n,(k+1)T/n[$). Note however that
	\begin{align*}
		\GE[|Y-Y^n|]=\GE[|Y-Y\circ T^n|]\leq L\,\GE[d(X^T, X^T\circ T^n)\wedge 2K],
	\end{align*}
	where $L>0$ and $K\cdot L$ are respectively a Lipschitz constant and bound of $Y$, $X^T$ is a canonical process, i.e. our $G$-\levy process, stopped at time T and $d$ is the Skorohod metric.
	
Now let us define the random variable $Z^n$ on $\Omega_T$ as follows
	\[
		Z^n(\omega):=\left\{ \begin{array}{ll}
 d(\omega,T^n( \omega))\wedge 2K & \textrm{if } \omega\in A_T,\\
0 & \textrm{otherwise},
\end{array} \right.
	\]
	where $A_T:=\{\omega\in \Omega_T\colon \omega \textrm{ has finite number of jumps in the interval }[0,T]\}.$	By Proposition \ref{prop_finite_activity} we know that $Z^n=d(X^T, X^T\circ T^n)\wedge 2K$ q.s. so the expectations of both random variables are equal
	\[
		\GE[d(X^T,X^T\circ T^n)\wedge 2K]=\GE[Z_n].
	\]
	Thus we can only consider paths with finite number of jumps. Fix then $\omega\in A_T$ having a finite number of jumps at time $0<r_1<\ldots<r_{m-1}< T$ and possibly a jump at $r_m:=T$. We can choose $n$ big enough such that $r_{i+1}-r_i\geq {T}/{n}$ for $i=0,\ldots, m-1$. Denote by $A_T^n$ the subset of $A_T$ containing all such $\omega$'s (i.e. with minimal distance between jumps larger or equal to ${T}/{n}$). We want to have an estimate of the Skorohod metric for $\omega\in A^n_T$. To obtain it we construct the piecewise linear function $\lambda^n$ as follows $\lambda^n(0)=0$, $\lambda^n(T)=T$, for each $k=1,\ldots,n-1$ define
	\[
		\lambda^n\left(\frac{kT}{n}\right) :=\left\{ \begin{array}{ll}
 \frac{kT}{n} & \textrm{if } r_i\notin \left] \frac{(k-1)T}{n},\frac{kT}{n}\right],\ i=1,\ldots,m ,\\
r_i & r_i\in \left] \frac{(k-1)T}{n},\frac{kT}{n}\right],\ i=1,\ldots,m.
\end{array} \right.
	\]
Moreover, let $\lambda^n$ be linear between these nods. By the construction $\|\lambda^n-Id\|_{\infty}\leq 2T/n$. Define $t_k:=\lambda^n({kT}/{n})$ for $k=0,\ldots,n$. Note that $\omega$ is continuous on $[t_k,t_{k+1}[$ and that $t_{k}\leq kT/n<t_{k+1},\ k=0,\ldots,n-1$.
	Then by definition of the Skorohod metric and property 2 in Lemma \ref{lem_modulus} we have	
	\begin{align*}
		d(\omega,T^n(\omega))\wedge 2K&=\left(\inf_{\lambda\in\Lambda}\max\{\|\lambda-Id\|_{\infty},\|T^n(\omega)-\omega\circ\lambda\|_{\infty}\}\right)\wedge2K\\
		&\leq \left(\|\lambda^n-Id\|_{\infty}+\|T^n(\omega)-\omega\circ\lambda^n\|_{\infty}\}\right)\wedge2K\\
		&\leq \left(\frac{2T}{n}+\max_{k=0,\ldots,n-1} \sup_{t\in[kT/n,(k+1)T/n[} |\omega\left(\frac{kT}{n}\right)-\omega\circ\lambda^n(t)|\right)\wedge 2K\\
		&\leq \left(\frac{2T}{n}+\max_{k=0,\ldots,n-1} \sup_{s,t\in[t_k,t_{k+1}[} |\omega(s)-\omega(t)|\right)\wedge 2K\\
		&\leq \left[\frac{2T}{n}+2w'_{\omega}\left(\frac{2T}{n}\right)\right]\wedge 2K.
	\end{align*}
	Thus we can define yet another bound $K^n$ as
    \[
        K^n(\omega):=\left\{
                       \begin{array}{ll}
                         \left(\frac{2T}{n}+2w'_{\omega}\left(\frac{2T}{n}\right)\right)\wedge 2K, & \textrm{if }\omega\in A_T^n,\\
                         2K, & \textrm{if }\omega\notin A_T^n.
                       \end{array}
                     \right.
    \]
    Then $K^n\geq Z^n$ and thus $\GE[Z^n]\leq \GE[K^n]$. We also have $K^n\downarrow 0$ on every $A^m_T$, $m$ is fixed. This follows from property 4 in Lemma \ref{lem_modulus}. Moreover we claim than $K^n$ is upper semi-continuous on every set $A_T^m$ for $m\leq n$. Firstly, note that the set $A^m_T$ is closed under the Skorohod topology. This is clear from the definition of the set: if $\{\omega^k\}_k\subset A^m_T$ then the distance between the jumps is $\geq {T}/{m}$ for each $k$. But if $\omega^k\to\omega$ then also $\omega$ must satisfy this property and hence it belong to $A^m_T\subset A^n_T$. Now note that by Lemma \ref{lem_modulus}, property 3, we have that $\omega\mapsto\left({2T}/{n}+2w'_{\omega}\left({2T}/{n}\right)\right)\wedge 2K$ is upper semi-continuous as a minimum of two upper semi-continuous functions and thus
        \begin{align*}
            \limsup_{k\to\infty} K^n(\omega^k)&= \limsup_{k\to\infty}  \left(\frac{2T}{n}+2w'_{\omega^k}\left(\frac{2T}{n}\right)\right)\wedge 2K\\&\leq \left(\frac{2T}{n}+2w'_{\omega}\left(\frac{2T}{n}\right)\right)\wedge 2K= K^n(\omega).
        \end{align*}
Thus $K^n$ is upper semi-continuous on each closed set $A^m_T,\ m\leq n$. 

We also claim that the sets $A_T^m$ are 'big' in the sense, that the capacity of the complement is decreasing to $0$. We prove it similarly to Proposition \ref{prop_finite_activity}. Note that
\[
	(A_T^m)^c=\{\omega\in\Omega_T\colon  \exists\, t,s\leq T,\ |t-s|<\frac{T}{m} \textrm{ and }\Delta \omega_t\neq 0,\ \Delta \omega_s\neq 0\}.
\]
For any $\theta \in \a^{\u}_{0,T}$ define the set
\[
	(A_T^{m,\theta})^c=\{\omega\in\Omega_T\colon  \exists\, t,s\leq T,\ |t-s|<\frac{T}{m} \textrm{ and }\Delta B^{0,\theta}_t(\omega)\neq 0,\ \Delta B^{0,\theta}_s(\omega)\neq 0\}.
\]
Then we have then by the representation of $c$, the fact that $\P^{\theta}$ is the law of $B.^{0,\theta}$ (which have jumps at times when Poisson process $M$ has a jump) and the properties of the Poisson process that
\begin{align*}
	c\left[(A^m_T)^c\right]&=\sup_{\theta\in\a^{\u}_{0,T}} \P^{\theta}\left[(A^m_T)^c\right]=\sup_{\theta\in\a^{\u}_{0,T}} \P_0\left[(A^{m,\theta}_T)^c\right]\\
	&\leq \P_0( \exists\, t,s\leq T,\ |t-s|<\frac{T}{m} \textrm{ and }\Delta M_t=\Delta M_s=1)=:\P_0(B^m_T).
\end{align*}
Since $M$ is a Poisson process, we get that $\P_0(\bigcap_{m=1}^{\infty} B^m_T)=0$. However, $B^m_T\supset B^{m+1}_T$, therefore by the continuity of the probability we get that $\P_0(B^m_T)\downarrow 0$ and consequently $c\left[(A^m_T)^c\right]\downarrow 0$.

Note that we will prove the assertion of our proposition if we use the following lemma (proof below).
\
\begin{lem}\label{lem_convergence_usc}
    Let $\{X_n\}_n$ be a sequence of non-negative uniformly bounded random variables on $\Omega_T$ such that there exists a sequence of closed sets $(F_m)_m$ having the following properties
    \begin{enumerate}
    \item $c(F_m^c)\to 0$ as $m\to\infty$.
    \item $X_n\downarrow 0$ on every $F_m$.
    \item  $X_n$ is upper semi-continuous on every $F_m,\ m\leq n$. 
\end{enumerate}     
Then $\GE[X_n]\to 0$.
\end{lem}

Applying this lemma to our sequence $\{K^n\}_n$ together with the closed sets $(A_T^m)_m$ we get that
\[
    \GE[|Y^n-Y|]\leq L\GE[ d(X^T,X^T\circ T^n)\wedge 2K]\leq L\GE[K^n]\to 0.\qedhere
\] 
\end{proof}

\begin{proof}[Proof of Lemma \ref{lem_convergence_usc}]
    Fix $\epsilon>0$. Let $M$ be the bound of all $X_n$. By the representation of the sublinear expectation we have
\begin{align*}
    \GE[X_n]&=\sup_{\theta\in\a^{\u}_{0,T}}\E^{\P^{\theta}}\left[X_n\right]=\sup_{\theta\in\a^{\u}_{0,T}}\int_0^M \P^{\theta}( X_n\geq t)dt \\&=\sup_{\theta\in\a^{\u}_{0,T}}\int_0^M \P^{\theta}[ (\{X_n\geq t\}\cap F_m)\cup (\{X_n\geq t\}\cap F_m^c)]dt\\
&\leq\sup_{\theta\in\a^{\u}_{0,T}}\int_0^M \P^{\theta}( \{ X_n|_{F_m}\geq t\}\cup F_m^c)dt
\leq\sup_{\theta\in\a^{\u}_{0,T}}\int_0^M \left[c(  X_n|_{F_m}\geq t)+c(F_m^c)\right]dt\\
&\leq\int_0^M c(  X_n|_{F_m}\geq t)dt+Mc(F_m^c).
\end{align*}
By the first property of sets $F_m$ we can choose $m$ big enough so that $c(F_m^c)\leq \frac{\epsilon}{2M}$.  Choose $n\geq m$. By the upper semi-continuity of $X_n$ on $F_m$ we get that each $\{X_n|_{F_m}\geq t\}$ is closed in the subspace topology on $F_m$. But $F_m$ is also a closed set in the Skorohod topology, thus $\{X_n|_{F_m}\geq t\}$ is also closed in it. Moreover, due to monotone convergence on $F_m$ we have that $\{X_n|_{F_m}\geq t\}\downarrow \emptyset$ as $n\uparrow \infty$. Thus by Lemma 7 in \cite{Denis_function_spaces} we get that $c(X_n|_{F_m}\leq t)\downarrow0$ as $n\uparrow \infty$ and we get the assertion of the lemma by applying monotone convergence theorem for the Lebesgue integral and choosing $n\geq m$ big enough, so that the integral is less then $\frac{\epsilon}{2}$. Thus
\[
	0\leq \GE[X_n]\leq \epsilon\quad \textrm{for } n \textrm{ big enough}.\qedhere
\]
\end{proof}
Now we are able to prove the main theorem using the reasoning by Denis \emph{et al.} as in Theorem 52 in \cite{Denis_function_spaces}, which is based on the Stone-Weierstrass theorem.
\begin{tw}\label{L_G_equal_L_c}
	The space $C_{b,lip}(\Omega_T)$ is dense in $C_b(\Omega_T)$ under the norm $\GE[|.|]$. Thus $L^1_G(\Omega_T)=\L^1_c(\Omega_T)$.
\end{tw}
\begin{proof}
	Fix $Y\in C_b(\Omega_T)$. We will prove that there exists a sequence $Y^n$ in 	$C_{b,lip}(\Omega_T)$ converging to $Y$ in $\GE[|.|]$-norm.
	
	Firstly, Ren proved that the family $\{\P^{\theta}\colon \theta \in \a^{\u}_{0,T}\}$ used to represent the sublinear expectation $\GE[.]$ is tight (see Lemma 3.4 in \cite{Ren}). Therefore by Prohorov's theorem (see e.g. Theorem 6 in \cite{Denis_function_spaces}) for each $n\in\N$ there exists a set $K_n$ which is compact in the Skorohod topology and $c(K_n^c)<1/n$.
	
	Note also that $C_{b,lip}(\Omega_T)$ is an subalgebra in $C_b(\Omega_T)$, which separates the points (the last claim is an easy consequence of the Tietze's extension theorem for Lipschitz functions, see Theorem 1.5.6. in \cite{Lipschitz}). Thus by the Stone-Weierstrass theorem for each compact set $K_n$ there exists a random random variable $Z^n$ on $K_n$, which is bounded and Lipschitz and 
	\[
		\sup_{\omega\in K_n} |Y(\omega)-Z^n(\omega)|<\frac{1}{n}\quad \textrm{and}\quad \sup_{\omega\in K_n} |Z^n(\omega)|\leq \sup_{\omega\in K_n} |Y(\omega)|.
	\]
	Once again using the Tietze's extension theorem, we may extend $Z^n$ to the whole $\Omega_T$ preserving the Lipschitz constant and the bound and we will denote this extension by $Y^n$. Note that $Y^n\in C_{b,lip}(\Omega_T)$ and that
	\[
		\sup_{\omega\in \Omega_T} |Y^n(\omega)|=\sup_{\omega\in K_n} |Z^n(\omega)|\leq \sup_{\omega\in K_n} |Y(\omega)|\leq \sup_{\omega\in \Omega_T} |Y(\omega)|=:M.
	\]
	Hence
	\begin{align*}
		\GE[|Y^n-Y|]&\leq \GE[|Y^n-Y|\I_{K_n}]+\GE[|Y^n-Y|\I_{K_n^c}]\leq 2M c(K_n^c)+\frac{1}{n} c(K_n)\\&\leq \frac{1}{n}(2M+1)\to 0.
	\end{align*}
	Therefore we proved that  $C_{b,lip}(\Omega_T)$ is dense in $C_b(\Omega_T)$. Thus the closure of  $C_{b,lip}(\Omega_T)$ under the norm $\GE[|.|]$ is exactly $\L^1_c(\Omega_T)$. However, by Proposition \ref{prop_lip_in_Lg} we know that $C_{b,lip}(\Omega_T)\subset L^1_G(\Omega_T)$ and by Remark 4 in \cite{Ren} we know that $L^1_G(\Omega_T)\subset \L^1_c(\Omega_T)$. Thus $L^1_G(\Omega_T)=\L^1_c(\Omega_T)$.
\end{proof}

The theorem above allows us  to use the characterization of the random variables in $\L^p_c(\Omega)$ in terms of their continuity and thickness of their tails. Namely, introduce the following definition.
\begin{defin}
	We will say that the random variable $Y\in L^0(\Omega)$ is quasi-continuous, if for all $\epsilon>0$ there exists an open set $O$ such that $c(O)<\epsilon$ and $Y|_{O^c}$ is continuous. For convenience, we will often use the abbreviation \emph{q.c}.
\end{defin}
It is well known that the following characterization of $\L^p_c(\Omega)$ (thus also $L^p_G(\Omega)$) holds (see Theorem 25 in \cite{Denis_function_spaces}).

\begin{prop}\label{prop_qL_are_in_Lip}
	For each $p\geq 1$ one has
\begin{align*}
	\L^p_c(\Omega)=L^p_G(\Omega)=\{Y\in \L^p\colon \lim_{n\to\infty}\GE[|Y|^p\I_{\{|Y|>n\}}]=0, \ Y \textrm{ has a q.c. version}\}.
\end{align*}
\end{prop}
Thanks to this proposition we will be sure that our integral  belongs to $L^2_G(\Omega)$ space.

\section{Definition of the \ito integral}\label{sec_integral}
By the definition of the $G$-\levy process there is a L\'evy-It\^o-type decomposition on the continuous part (i.e. generalized $G$-Brownian motion) and pure-jump process. As it is widely known, there is a good definition of the \ito integral w.r.t. $G$-Brownian motion, so we will deal only with the pure-jump part.

We introduce the following random measure: for any $0\leq t<s$ and $A\in \B(\r0)$
\[
	X(]t,s],A):=\sum_{t<u\leq s}\I_{A}(\Delta X_u),\ q.s.
\]

\begin{rem} Note that this random measure is well-defined q.s. thanks to the finite activity property. Moreover, it is really a random measure, i.e. is countably additive. This is not true if one would like to compensate it by factor $\GE[X(]t,s],A)]$, as a a function $A\to \GE[X(]t,s],A)]$ is usually not additive, as one can easily check. Though it is not such a big obstacle for defining the \ito integral, it shows a big difference to the ordinary Poisson random measures, which can  always be compensated (compare \cite{applebaum}). Moreover, it shows that for disjoints sets $A_1$ and $A_2$ the random variables $X(]t,s],A_1)$ and $X(]t,s],A_2)$ are NOT independent under $\GE[.]$\footnote{This is not surprising, as there is already a good characterization of the random variables mutually independent of each other under sublinear expectations, see \cite{Hu_indep}.}.
\end{rem}

Let us now introduce the set of simple integrands.
\begin{defin}
	Let $\H^S_G([0,T]\times\r0)$ be a collection of all processes defined on $[0,T]\times \r0\times\Omega$ of the form
	\begin{equation}\label{eq_K_representation}
			K(u,z)(\omega)=\sum_{k=1}^{n-1}\,\sum_{l=1}^m\, F_{k,l}(\omega)\,\I_{]t_{k},t_{k+1}]}(u)\,\psi_l(z),\ n,m\in\N,
	\end{equation}
	where $0\leq t_1<\ldots<t_n\leq T$ is the partition of $[0,T]$, $\{\psi_l\}_{l=1}^m\subset C_{b,lip}(\R^d)$ are functions with disjoint supports s.t. $\psi_l(0)=0$ and $F_{k,l}=\phi_{k,l}(X_{t_1},\ldots, X_{t_k}-X_{t_{k-1}})$, \linebreak$\phi_{k,l}\in C_{b,lip}(\R^{d\times k})$.
	We introduce two norms on this space
	\[
		\|K\|_{\H_G^p([0,T]\times\r0)}:=\GE\left[\int_0^T\,\sup_{v\in\v}\,\int_{\r0}|K(u,z)|^pv(dz)du\right]^{1/p},\quad p=1,2.
	\]
	Note that the norms are well defined on $\H^S_G([0,T]\times\r0)$.
\end{defin}
\begin{defin}\label{def_ito_int}
	Let $0\leq s<t\leq T$. The \ito integral of $K\in \H^S_G([0,T]\times\r0)$ w.r.t.\ jump measure $X$ is a random variable defined as
	\[
		\int_s^t\int_{\r0}\, K(u,z) X(du,dz):=\sum_{s<u\leq t}\, K(u,\Delta X_u),\ q.s.
	\]
	If $s=0,\ t=T$ we will denote also the integral as an operator $I$.
\end{defin}
\begin{tw}\label{tw_integral_continuous}
	\ito integral $I$ is a continuous linear operator from $\H^S_G([0,T]\times\r0)$ equipped with the norm $\|.\|_{\H^p_G([0,T]\times\r0)}$ to $\L^p$ for $p=1,2$.
\end{tw}
\begin{proof}
	We will prove the theorem only for $p=2$ as the case $p=1$ follows the same argument. We will utilize Corollary \ref{col_representation_L_G}. Let $K$ has representation as in eq. (\ref{eq_K_representation}), i.e.
		\[
		K(u,z)(\omega)=\sum_{k=1}^{n-1}\,\sum_{l=1}^m\, F_{k,l}(\omega)\,\I_{]t_{k},t_{k+1}]}(u)\,\psi_l(z),
	\]
	where $0\leq t_1\leq\ldots\leq t_n\leq T$, $\{\psi_l\}_{l=1}^m\subset C_{b,lip}(\R^d)$ are functions with disjoint supports s.t. $\psi_l(0)=0$  and  $F_{k,l}=\phi_{k,l}(X_{t_1},\ldots, X_{t_k}-X_{t_{k-1}})$, $\phi_{k,l}\in C_{b,lip}(\R^{d\times k})$. By the corollary and the definition of the \ito integral we have
	\begin{align}\label{eq_ito_isometry1}
		\GE&\left[\left(\int_0^T\int_{\r0}\, K(u,z) X(du,dz)\right)^2\right]= \sup_{\theta\in\a^{\u}_{0,T}}\,\E^{\P^{\theta}}\left[\left(\sum_{0<u\leq T} K(u,\Delta X_u)\right)^2\right]\notag\\
		=& \sup_{\theta\in\a^{\u}_{0,T}}\,\E^{\P^{\theta}}\left[\left(\sum_{0<u\leq T} \sum_{k=1}^{n-1}\,\sum_{l=1}^m\, \phi_{k,l}(X_{t_1},\ldots, X_{t_k}-X_{t_{k-1}})\,\I_{]t_{k},t_{k+1}]}(u)\,\psi_l(\Delta X_u)\right)^2\right]\notag\displaybreak[1]\\
		=& \sup_{\theta\in\a^{\u}_{0,T}}\,\E^{\P_0}\left[\left(\sum_{0<u\leq T} \sum_{k=1}^{n-1}\,\sum_{l=1}^m\, \phi_{k,l}(B^{0,\theta}_{t_1},\ldots, B_{t_k}^{t_{k-1},\theta})\,\I_{]t_{k},t_{k+1}]}(u)\,\psi_l(\Delta B^{0,\theta}_u)\right)^2\right]\notag\\
		=& \sup_{\theta\in\a^{\u}_{0,T}}\,\E^{\P_0}\left[\left(\sum_{0<u\leq T} \sum_{k=1}^{n-1}\,\sum_{l=1}^m\, F^{\theta}_{k,l}\,\I_{]t_{k},t_{k+1}]}(u)\,\psi_l(\theta^d(u,\Delta N_u))\right)^2\right],
		\end{align}
		where $F^{\theta}_{k,l}:= \phi_{k,l}(B^{0,\theta}_{t_1},\ldots, B_{t_k}^{t_{k-1},\theta})$.
	Define a predictable process $K^{\theta}(u,z)$ as
	\begin{align*}
		K^{\theta}(u,z)&:=\sum_{k=1}^{n-1}\,\sum_{l=1}^m\, F^{\theta}_{k,l}\,\I_{]t_{k},t_{k+1}]}(u)\,\psi_l(\theta^d(u,z)).
	\end{align*}
	Then we can write eq. (\ref{eq_ito_isometry1}) as
	\[
		\GE\left[\left(\int_0^T\int_{\r0}\, K(u,z) X(du,dz)\right)^2\right]= \sup_{\theta\in\a^{\u}_{0,T}}\,\E^{\P_0}\left[\left(\int_0^T\int_{\r0}\, K^{\theta}(u,z) N(du,dz)\right)^2\right],
	\]
	\[
		=\sup_{\theta\in\a^{\u}_{0,T}}\,\E^{\P_0}\left[\left(\int_0^T\int_{\r0}\, K^{\theta}(u,z) \tilde{N}(du,dz)+\int_0^T\int_{\r0}\, K^{\theta}(u,z)\mu(dz)du\right)^2\right],
	\]
	where $N(du,dz)$ and $\tilde{N}(du,dz)$ are respectively the Poisson random measure and the compensated Poisson measure associated with the \levy process with the \levy triplet $(0,0,\mu)$. Using standard inequalities we get hence:
	\begin{align}
		&\GE\left[\left(\int_0^T\int_{\r0}\, K(u,z) X(du,dz)\right)^2\right]\notag\\&\leq 2\sup_{\theta\in\a^{\u}_{0,T}}\,\left\{\E^{\P_0}\left[\left(\int_0^T\int_{\r0}\, K^{\theta}(u,z) \tilde{N}(du,dz)\right)^2+\left(\int_0^T\int_{\r0}\, K^{\theta}(u,z)\mu(dz)du\right)^2\right]\right\}\notag\\
		&\leq 2\sup_{\theta\in\a^{\u}_{0,T}}\left\{\int_0^T\int_{\r0}\E^{\P_0}\left[\left(K^{\theta} (u,z)\right)^2\right] \mu(dz)du+T\int_0^T\int_{\r0} \E^{\P_0}\left[\left(K^{\theta}(u,z)\right)^2\right]\mu(dz)du\right\}\displaybreak[1]\notag\\
		&=C_T \sup_{\theta\in\a^{\u}_{0,T}}\int_0^T\!\int_{\r0}\!\E^{\P_0}\left[\left(\sum_{k=1}^{n-1}\sum_{l=1}^m\,F^{\theta}_{k,l}\,\I_{]t_{k},t_{k+1}]}(u)\,\psi_l(\theta^d(u,z))\right)^2\right] \mu(dz)du,\notag
	\end{align}	
	where $C_T:=2(T+1)$. Note that the intervals $]t_k,t_{k+1}]$ are mutually disjoint, just as the supports of $\psi_l$, hence 
		\begin{align}\label{eq_ito_isometry2}
		&\GE\left[\left(\int_0^T\int_{\r0}\, K(u,z) X(du,dz)\right)^2\right]\notag\\
		&\leq C_T\sup_{\theta\in\a^{\u}_{0,T}}\sum_{k=1}^{n-1}\sum_{l=1}^m\int_0^T\int_{\r0}\E^{\P_0}\left[(F^{\theta}_{k,l})^2\,\I_{]t_{k},t_{k+1}]}(u)\psi^2_l(\theta^d(u,z))\right] \mu(dz)du\notag\\
		&=C_T\sup_{\theta\in\a^{\u}_{0,T}}\, \sum_{k=1}^{n-1}\sum_{l=1}^m\,\int_{t_k}^{t_{k+1}}\, \E^{\P_0}\left[\int_0^T (F^{\theta}_{k,l})^2\,\int_{\r0}\psi^2_l(\theta^d(u,z)) \mu(dz)\right]du.
	\end{align}	
	By the assumptions on the process $\theta^d$, we know that for a.a. $\omega$ and a.e. $u$ function $z\mapsto \theta^d(u,z)(\omega)$ is equal to $g_v$ for some $v\in\v$. Hence we can define a random measure $\pi^{\theta}_u$ as $\pi^{\theta}_u(\omega):=v$ if $\theta^d(u,.)(\omega)=g_v$. Note that $\pi^{\theta}_u \in \v$, $\P_0-a.s.$ for a.e. $u$ and for every set $A\in\B(\R^d)$ the function $(u,\omega)\mapsto\pi^{\theta}(A)$ is $\B([0,T])\otimes\F_T$-measurable (as the random field $\theta^d$ is predictable). Hence, we can transform \eqref{eq_ito_isometry2} to get
	\begin{align*}
		\GE&\left[\left(\int_0^T\int_{\r0}\, K(u,z) X(du,dz)\right)^2\right]\notag\\
		&\leq C_T\sup_{\theta\in\a^{\u}_{0,T}}\, \sum_{k=1}^{n-1}\,\sum_{l=1}^m\,\int_{t_k}^{t_{k+1}}\, \E^{\P_0}\left[\phi^2_{k,l}(B^{0,\theta}_{t_1},\ldots, B_{t_k}^{t_{k-1},\theta})\int_{\r0}\psi_l^2(z)\pi^{\theta}_u(dz)(.)\right]du\notag
		\\
		&=C_T\sup_{\theta\in\a^{\u}_{0,T}}\, \E^{\P_0}\left[\sum_{k=1}^{n-1}\,\int_{t_k}^{t_{k+1}}\,\sum_{l=1}^m\, \phi^2_{k,l}(B^{0,\theta}_{t_1},\ldots, B_{t_k}^{t_{k-1},\theta})\int_{\r0}\psi_l^2(z)\pi^{\theta}_u(dz)(.)du\right]\notag
		\\
		&\leq C_T\sup_{\theta\in\a^{\u}_{0,T}}\, \E^{\P_0}\left[\sum_{k=1}^{n-1}\,\int_{t_k}^{t_{k+1}}\sup_{v\in\v}\,\sum_{l=1}^m\, \phi^2_{k,l}(B^{0,\theta}_{t_1},\ldots, B_{t_k}^{t_{k-1},\theta})\int_{\r0}\psi_l^2(z)v(dz)du\right]\notag
		\\
		&\leq C_T\sup_{\theta\in\a^{\u}_{0,T}}\, \E^{\P_0}\left[\int_{0}^{T}\sup_{v\in\v}\int_{\r0}\sum_{k=1}^{n-1}\sum_{l=1}^m\phi^2_{k,l}(B^{0,\theta}_{t_1},\ldots, B_{t_k}^{t_{k-1},\theta})\I_{]t_k,t_{k+1}]}(u)\psi^2_l(z)\,v(dz) du\right]\notag\\
		&=C_T\,\GE\left[\,\int_{0}^{T}\sup_{v\in\v}\,\int_{\r0}\, K^2(u,z)\,v(dz) du\right].
	\end{align*}
	We stress that the integral w.r.t. the random measure $\pi^{\theta}_u$ in the second and third line is well defined and it is also $\B([0,T])\otimes\F_T$-measurable thanks to the measurability of $\pi^{\theta}_u$. Hence, all other integrals also make sense.
\end{proof}
\begin{tw}\label{tw_integral_in_LpG}
    For every $K\in \H^S_G([0,T]\times\r0)$ we have that $\int_s^t\int_{\r0}\, K(u,z) X(du,dz)$ is an element of both $L^1_G(\Omega)$ and $L^2_G(\Omega)$.
\end{tw}
\begin{proof}
    First note that $L^2_G(\Omega_T)\subset L^1_G(\Omega_T)$, therefore it is sufficient to prove that the integral belong to the smaller space. By the linearity of $L^2_G(\Omega)$ it also suffices to prove the assertion of the theorem for $K$ of the form 
    \[
        K(u,z):=\psi(z),\quad \psi\in C_{b,lip}(\R^d).
    \]
    Note that we can take $K$ deterministic, because for every $X\in L^2_G(\Omega)$ and $\xi\in Lip(\Omega)$, $X\cdot \xi\in L^2_G(\Omega)$ due to the boundedness of $\xi$. Thus it suffices to prove that 
    \[
        \sum_{s<u\leq t} \psi (\Delta X_u)=:Y\in L^2_G(\Omega).
    \]
    We will use Proposition \ref{prop_qL_are_in_Lip} noting that by Theorem \ref{tw_integral_continuous} the integral $Y$ is in $\L^2$ space. 
    
    Firstly, we will prove that $Y$ has a quasi-continuous version. Introduce a random variable $Z:=  \sum_{s\leq u\leq t} \psi (\Delta X_u)$ and a set
     \begin{multline*} 
     	A_n:=\{\omega\in\Omega\colon \omega \textrm{ has at most } n\textrm{ jumps in the interval } ]s,t[\\ \textrm{ and no jumps in both } ]s-1/n,s[ \textrm{ and } ]t,t+1/n[\,\}.
     \end{multline*} 
     Fix a sequence $(\omega^m)_m\subset A_n$ converging to $\omega$ in the Skorohod topology. By the definition of the Skorohod metric it is easy to see that if $\Delta \omega_u\neq 0$, $u\in]s,t[$, then there exists a sequence $(u_m)_m\subset ]s,t[$ converging to $u$ s.t. \begin{equation}\label{eq_conv_omega_u}
     \Delta\omega^m_{u_m}\to \Delta \omega_u. 
\end{equation} 
Conversely, if there exists a sequence $(u_m)_m\subset ]s,t[$ converging to $u\in ]s,t[$ and such that $\Delta \omega^m_{u_m}\neq 0$ for almost all $m$, then $\Delta \omega^m_{u_m}\to \Delta \omega_u$ (which might be equal to 0). By this we conclude that $\omega$ has at most $n$ jumps in the interval $]s,t[$. Similarly, we claim that $\omega$ can't have any jumps in the intervals $]s-1/n,s[$ and $]t,t+1/n[$. Thus, $A_n$ is a closed set under Skorohod topology. 
     
Moreover, by the definition of $A_n$ we have that jumps of $\omega^m$ can neither escape the interval $[s,t]$ nor enter it as $n$ goes to infinity. Thus by \eqref{eq_conv_omega_u} and the continuity of $\psi$ we get
\[
    \sum_{s\leq u\leq t}\psi(\Delta \omega_u^m)\to \sum_{s\leq u \leq t}\psi(\Delta \omega_u).
\]
Hence, $Z$ is continuous on $A_n$.

We prove now that the complement of $A_n$ has a small capacity. By the same argument as in the Proposition \ref{prop_finite_activity}, we can show that $X$ has at least $n$ jumps in the interval $]\alpha,\beta[$ implies that the Poisson process $M$ needs to have also at least $n$ jumps in the same interval and the capacity of set $A_n^c$ might be dominated in the following manner
\begin{align*}
    c(A_n^c)\leq &\P_0(M \textrm{ has at least }n\textrm{ jumps in the interval }]s,t[)\\&+\P_0(M \textrm{ has at least }1\textrm{ jump in the interval }]s-1/n,s[)\\&+\P_0(M \textrm{ has at least }1\textrm{ jump in the interval }]t,t+1/n[)\downarrow 0.
\end{align*}
Moreover, $A_n^c$ is open as the complement of a closed set. Hence, we conclude that $Z$ is quasi-continuous. 

It remains to show that $Z=Y$ q.s. This is however easy, since
\[
	c(Y\neq Z)=c(\Delta X_s\neq 0)\leq \P_0(\Delta M_s\neq 0)=0.
\]

The 'uniform integrability condition' might be proved in the similar manner. Let $K$ be a bound of $\psi$. Without the loss of generality we may assume that $K=1$. 
Note that  for any $\theta \in \a^{\u}_{0,\infty}$ we have the following inclusion
\begin{align*}
    &\left\{\left|\sum_{s<u\leq t}\psi(\theta^d(u,\Delta N_u))\right|>n\right\}\subset\left\{\sum_{s<u\leq t}\left|\psi(\theta^d(u,\Delta N_u))\right|>n\right\}\\&\quad\subset \{N \textrm{ has at least } n\textrm{ jumps in }]s,t]\}\subset\{M \textrm{ has at least } n\textrm{ jumps in }]s,t]\}=:B_n,
\end{align*}
as the sum of jumps grows only at jump times and only by a value bounded by $1$. Introduce
\[
    C_n:=B_{n}\setminus B_{n+1}=\{M \textrm{ has } n\textrm{ jumps in the interval }]s,t]\}
\]
Hence we have the estimate
\begin{align*}
    \GE&\left[|Y|^2\I_{\{|Y|>n\}}\right]=  \sup_{\theta\in\a^{\u}_{0,T}}\E^{\P_0}\left[\left|\sum_{s<u\leq t}\psi(\Delta B^{0,\theta}_u)\right|^2\I_{\left\{\left|\sum_{s<u\leq t}\psi(\Delta B^{0,\theta}_u)\right|>n\right\}}\right]\\
    &=\sup_{\theta\in\a^{\u}_{0,T}}\E^{\P_0}\left[\left|\sum_{s<u\leq t}\psi(\theta^d(u,\Delta N_u))\right|^2\I_{\left\{\left|\sum_{s<u\leq t}\psi(\theta^d(u,\Delta N_u))\right|>n\right\}}\right]\\
&\leq \sup_{\theta\in\a^{\u}_{0,T}}\E^{\P_0}\left[\sum_{s<u\leq t}\left|\psi(\theta^d(u,\Delta N_u))\right|^2\I_{B_n}\right]
\\
&\leq \sum_{m=n}^{\infty}\sup_{\theta\in\a^{\u}_{0,T}}\E^{\P_0}\left[\sum_{s<u\leq t}\left|\psi(\theta^d(u,\Delta N_u))\right|^2\I_{C_m}\right]
\leq \sum_{m=n}^{\infty}\sup_{\theta\in\a^{\u}_{0,T}}\E^{\P_0}\left[m^2\I_{C_m}\right]
\\&=\sum_{m=n}^{\infty}\, m^2\P_0(C_m), \to 0\textrm{ as }n\to\infty,
\end{align*}
because the Poisson random variable has second moment finite and the number of jumps of the Poisson process in the fixed intervals is Poisson-distributed.

We conclude the proof by noting that $Y$ satisfies the characterization in Proposition \ref{prop_qL_are_in_Lip}, thus belongs to the space $L^2_G(\Omega)$ (and therefore also to $L^1_G(\Omega)$).
\end{proof}
\begin{cor}
	Let $\H^p_G([0,T]\times\r0)$ denote the topological completion of $\H^S_G([0,T]\times\r0)$ under the norm $\|.\|_{\H_G^p([0,T]\times\r0)},\ p=1,2$. Then \ito integral can be continuously extended by the continuity of the operator $I$ to the whole space $\H^p_G([0,T]\times\r0),\ p=1,2$. Moreover, by Theorem \ref{tw_integral_in_LpG} we know that the extended operator takes value in $L^p_G(\Omega_T),\ p=1,2$. Lastly, the formula from Definition \ref{def_ito_int} still holds for all $K\in\hgTR$.
\end{cor}
\begin{rem}
	Since the jump measure is not compensated, there is no reason to expect that the expectation of such an integral should be $0$. Moreover, it is easy to check that in general the \itolevy integral is not a symmetric random variable. As the consequence the nature of this integral is significantly different from \ito integral w.r.t. $G$-Brownian motion.
\end{rem}

\section{$G$-\itolevy processes. \ito formula}\label{sec_ito_form}
In this section we will introduce $G$-\ito\levy processes. Assume $\u$ to be of the form\break $\u:=\v\times\{0\}\times\q$, i.e. there is no drift-uncertainty. Assume moreover that the set $\q$ is bounded and convex. Then the $G$-\levy process $X$ associated with $\u$ might be represented as $X:=B+L$, where $B$ is a $G$-Brownian motion associated with $\q$ and $L$ is a pure-jump $G$-\levy process associated with $\v$.\footnote{Formally, we should introduce new operators $G^c$ and $G^d$ which would produce $G^c$-Brownian motion and pure jump $G^d$-\levy process. However, we think that in this paper this slight abuse of notation does not lead to any confusion, so we will keep it.}

Recall the following notation
\[
	x\cdot y:= x^Ty,\quad |x|:=\sqrt{x\cdot x},\quad \textrm{and } \gamma\colon \beta:=\tr(\gamma\beta),\quad |\gamma|:=\sqrt{\gamma\colon\gamma},
\]
where $x,y,\in\R^d$, $\gamma,\beta\in \s^d$ ($\s^d$ is the space of all $d\times d$-dimensional symmetric matrices).

We will also use the following definition: the process $Z$ taking values in a metric space $(\mathcal{X},d)$ is an elementary process, if it has the form
\[
	Z_t=\sum_{n=1}^N\, \phi_n(X_{t_1},\ldots,X_{t_n})\I_{]t_{n-1},t_n]},
\]
where $0\leq t_1<\ldots t_N<\infty$ and $\phi_n\colon\R^{d\times n}\to \mathcal{X}$ is Lipschitz continuous and bounded. 

Define the following spaces
\begin{enumerate}
	\item Let $\H_G^2(0,T)$ denote the completion of all $\R^d$-valued elementary processes under the norm
	\[
		\|Z\|^2_{\hgT}:=\GE\left[\int_0^TZ_sZ_s^T\colon d\qB_s\right].
	\]
	For a process $Z\in\hgT$ one can define the stochastic integral denoted by\break $\int_0^t Z_s\cdot dB_s$. Note that usually one uses different elementary processes (with random variables which are cylinder functions of the $G$-Brownian motion and not a $G$-\levy process). However we can easily generalize the \ito integral for this larger class of integrands, because the increment of the $G$-Brownian motion is independent of the past of the pure-jump $G$-\levy process. \footnote{If $\q$ is also bounded away from $0$, then there exist constants $0<A\leq B$ such that $At\cdot Id_d\leq \qB_t\leq Bt\cdot Id_d$ and the norm $\H_G^2(0,T)$ is equivalent to the following norm: $\GE\left[\int_0^T|Z_s|^2ds\right]^{1/2}$. }
	\item Let $\mg1T$ denote the completion of all $\R$-valued elementary processes under the norm
	\[
		\|\eta\|_{\mg1T}:=\GE\left[\int_0^T|\eta_s|ds\right].
	\]
	For a process $\eta\in\mg1T$ one can define the following integral $\int_0^t \eta_s ds$.
	\item Let $\Mg1T$ denote the completion of all $\s^d$-valued elementary processes under the norm
	\[
		\|\eta\|_{\Mg1T}:=\GE\left[\int_0^T|\eta_s|ds\right].
	\]
	For a process $\eta\in\Mg1T$ one can define the following integral $\int_0^t \eta_s\colon d\qB_s$.
\end{enumerate}
\begin{defin}
	The process $Y=(Y^1,\ldots,Y^m)$ is called an $m$-dimensional $G$-\itolevy process, if there exist processes $Z^i\in \hgT$, $\alpha^i\in \mg1T$, $\beta^i\in\Mg1T$  and $K^i\in \hgTR$, $i=1,\ldots,m$ such that for all $t\in [0,T]$
	\begin{equation}\label{eq_ito_levy_processes}
		Y^i_t=Y^i_0+\int_0^t \alpha^i_s ds+ \int_0^t \beta^i_s:d\qB_s+\int_0^tZ^i_s \cdot dB_s  +\int_0^t\int_{\r0}\, K^i(s,z) L(ds,dz).\ q.s.
	\end{equation}
\end{defin}

\begin{tw}[\ito formula]\label{tw_ito_formula}	
	Let $Y$ be a $G$-\itolevy process in $\R^m$ with representation (\ref{eq_ito_levy_processes}). Let $f\in C_b^2(\R^m)$. Then $f(Y_t)$ is also a $G$-\levy-\ito process with the representation
	\begin{align*}
		f(Y_t)&=f(Y_0)+ \sum_{i=1}^m\int_0^t\frac{\partial f}{\partial x_i} (Y_s)\alpha^i_s ds+ \sum_{i=1}^m\int_0^t\frac{\partial f}{\partial x_i} (Y_s)\beta^i_s\colon d\qB_s\\&+\frac{1}{2}\sum_{i,j=1}^m\int_0^t\frac{\partial^2 f}{\partial x_i \partial x_j} (Y_s)Z^i_s(Z^j_s)^T\colon d\qB_s+\sum_{i=1}^m\int_0^t\frac{\partial f}{\partial x_i}(Y_s)Z^i_s\cdot dB_s\\&+\int_0^t\int_{\r0}\left[f(Y_{s-}+K(s,z))-f(Y_{s-})\right]L(ds,dz), q.s.
	\end{align*}
	where $K:=(K^1,\ldots,K^m)$.
\end{tw}
\begin{proof}
	Firstly, define the following random times
	\[
		\tau_0=0,\quad \tau_n:=\inf\{t>\tau_{n-1}\colon 0\neq\Delta X_t(=\Delta L_t) \},\ n=1,2,\ldots
	\]
	Each $\tau_n$ is a stopping time w.r.t. filtration generated by the canonical process $X$. Due to finite activity $\tau_n\uparrow\infty$ q.s. Thus we have
	\begin{align}\label{eq_ito_form0}
		f(Y_t)-f(Y_0)&=\sum_{n=1}^{\infty} \left[f(Y_{t\wedge \tau_n})-f(Y_{t\wedge \tau_{n-1}})\right]\notag\\
		&=\sum_{n=1}^{\infty} \left[f(Y_{t\wedge \tau_n-})-f(Y_{t\wedge \tau_{n-1}})\right]+\sum_{n=1}^{\infty} \left[f(Y_{t\wedge \tau_n})-f(Y_{t\wedge \tau_{n}-})\right],\ q.s.
	\end{align}
	Note that the second sum might be written as
		\begin{align}\label{eq_ito_form3}
		\sum_{n=1}^{\infty} \left[f(Y_{t\wedge \tau_n})-f(Y_{t\wedge \tau_{n}-})\right]&=\sum_{n=1}^{\infty} \left[f(Y_{t\wedge \tau_n-}+K(t\wedge \tau_n,\Delta L_{t\wedge\tau_n})-f(Y_{t\wedge \tau_{n}-})\right]\notag\\
		&=\int_0^t\int_{\r0} \left[f(Y_{s-}+K(s,z))-f(Y_{s-})\right]L(ds,dz),\ q.s.
	\end{align}
	The first sum is more complicated and one has to be cautious by dealing with stopping times here. Fix $n$ and introduce the process $Y_{n,t}:=(Y^1_{n,t},\ldots,Y^m_{n,t})$ where $Y_{n,t}^i$ is defined as
	\[
		 Y^i_{n,t}=Y^i_{t\wedge\tau_{n-1}}+\int_0^t\alpha^i_s\I_{]\tau_{n-1},\tau_n]}ds+\int_0^t\beta^i_s\I_{]\tau_{n-1},\tau_n]}\colon d\qB_s+\int_0^tZ^i_s\I_{]\tau_{n-1},\tau_n]}\cdot dB_s.
	\]
	Note that the integrands may fall out of their spaces, when multiplied by a factor $\I_{]\tau_{n-1},\tau_n]}$, as the multiplied integrands might lose so-called quasi-continuity (see \cite{Denis_function_spaces} and \cite{Song_hitting} for the discussion of this problem). This general problem is not an obstacle for us, if we use the definition of  integrals which does not assume the quasi-continuity of the integrand. Such a definition was introduced by Li and Peng in \cite{Li_Peng_stopping} and we can utilize it immediately. They also gave the \ito formula for such processes (Theorem 5.4), which we apply now to the process $Y^n_t$ and $f$. Thus
	\begin{align}\label{eq_ito_form1}
		f(Y_{n,t})=&f(Y_{t\wedge\tau_{n-1}})+\sum_{i=1}^m\int_0^t\frac{\partial f}{\partial x_i} (Y_{n,s})\alpha^i_s\I_{]\tau_{n-1},\tau_n]} ds\notag\\
		&+ \sum_{i=1}^m\int_0^t\frac{\partial f}{\partial x_i} (Y_{n,s})\beta^i_s\I_{]\tau_{n-1},\tau_n]}\colon d\qB_s+\sum_{i=1}^m\int_0^t\frac{\partial f}{\partial x_i}(Y_{n,s})Z^i_s\I_{]\tau_{n-1},\tau_n]}\cdot dB_s\notag\\
		&+\frac{1}{2}\sum_{i,j=1}^m\int_0^t\frac{\partial^2 f}{\partial x_i \partial x_j} (Y_{n,s})Z^i_s(Z^j_s)^T\I_{]\tau_{n-1},\tau_n]}\colon d\qB_s,\ q.s.
	\end{align}

Firstly, notice that by Lemma 4.3 in \cite{Li_Peng_stopping} one has
\[
			Y^i_{n,t}=Y^i_{t\wedge\tau_{n-1}}+\int_{]t\wedge\tau_{n-1},t\wedge\tau_n]}\left[\alpha^i_sds+\beta^i_s\colon d\qB_s+Z^i_s\cdot dB_s\right],\quad q.s.
\]
Thus $Y^i_{n,t}=Y^i_t$ q.s. on $[\tau_{n-1},\tau_n[$, $i=1,\ldots,m$ and hence we can rewrite (\ref{eq_ito_form1}) as
	\begin{align}\label{eq_ito_form2}
		f&(Y_{t\wedge\tau_n-})=f(Y_{t\wedge\tau_{n-1}})+\sum_{i=1}^m\int_0^t\frac{\partial f}{\partial x_i} (Y_{s})\alpha^i_s\I_{]\tau_{n-1},\tau_n]} ds\notag\\
		&+ \sum_{i=1}^m\int_0^t\frac{\partial f}{\partial x_i} (Y_s)\beta^i_s\I_{]\tau_{n-1},\tau_n]}\colon d\qB_s+\sum_{i=1}^m\int_0^t\frac{\partial f}{\partial x_i}(Y_{s})Z^i_s\I_{]\tau_{n-1},\tau_n]}\cdot dB_s\notag\\
		&+\frac{1}{2}\sum_{i,j=1}^m\int_0^t\frac{\partial^2 f}{\partial x_i \partial x_j} (Y_{s})Z^i_s(Z^j_s)^T\I_{]\tau_{n-1},\tau_n]}\colon d\qB_s,\ q.s.
	\end{align}
Now taking a sum in (\ref{eq_ito_form2}) we get that
\begin{multline}\label{eq_ito_form4}
	\sum_{n=1}^{\infty} \left[f(Y_{t\wedge \tau_n-})-f(Y_{t\wedge \tau_{n-1}})\right]=\sum_{i=1}^m\int_0^t\frac{\partial f}{\partial x_i} (Y_s)\alpha^i_s ds+ \sum_{i=1}^m\int_0^t\frac{\partial f}{\partial x_i} (Y_s)\beta^i_s\colon d\qB_s\\+\frac{1}{2}\sum_{i,j=1}^m\int_0^t\frac{\partial^2 f}{\partial x_i \partial x_j} (Y_s)Z^i_s(Z^j_s)^T\colon d\qB_s+\sum_{i=1}^m\int_0^t\frac{\partial f}{\partial x_i}(Y_s)Z^i_s\cdot dB_s,\ q.s.
\end{multline}
Combining eq. (\ref{eq_ito_form0}), (\ref{eq_ito_form3}) and (\ref{eq_ito_form4}) we get the assertion of the theorem.
\end{proof}

\section{Diffusions with jump uncertainty}\label{sec_diffusions}
At the end of this section, we will establish the SDE's and BSDE's w.r.t. $G$-\levy processes. Once again we assume that $\u=\v\times\{0\}\times\q$ with $\q$ bounded and convex. Thus the quadratic variation $\qB_t$ might be dominated by $M\cdot t\cdot Id_d$ for some constant $M$, or more specifically, the quadratic covariation $\langle B^i,B^j\rangle_t$ might be dominated by $M^{i,j}\cdot t$.

We will follow the idea presented by Peng in \cite{Peng_skrypt}, Chapter V.

Let us introduce the new norm on the integrands: for a $\R^n$-dimensional process $Z$ define
\[
	\|Z\|^p_{\hMTp}:=\int_0^T\GE[|Z_t|^p]dt,\ p\geq1.
\]
The completion of the space of $n$-dimensional elementary processes under this norm will be denoted as $\hMTp$. Note that
\[
	 \GE\left[\int_0^T|Z_t|^pdt\right]\leq \int_0^T\GE[|Z_t|^p]dt,
\]
thus appropriate integrals will be always well defined.

Similarly, we need to adjust the space of integrands for the jump measure.  Let $\hhR$ denote the completion of all $\H^S_G([0,T]\times\r0)$ under the norm
	\[
		\|K\|^2_{\hhR}:=\int_0^T\,\GE\left[\sup_{v\in\v}\,\int_{\r0}K^2(u,z)v(dz)\right]du.
	\]
	
\subsection{SDE's driven by $G$-\levy processes}
We will consider the following SDE driven by the $d$-dimensional $G$-Brownian motion $B$ and the $d$-dimensional pure jump $G$-\levy process $L$
\begin{equation}\label{eq_sde}
	dY^i_s=b^i(s,Y_s)ds+h^i(s,Y_{s})\colon d\qB_s+\sigma^i(s,Y_{s})\cdot dB_s+\int_{\r0}K^i(s,Y_{s-},z)L(dz,ds),
\end{equation}
where $i=1,\ldots,n$, $Y=(Y^1,\ldots,Y^n)$. Denote $b=(b^1,\ldots,b^n)$, $h=(h^1,\ldots,h^n)$, $\sigma=(\sigma^1,\ldots,\sigma^n)$ and $K=(K^1,\ldots,K^n)$.

We will work under following standard assumptions.
\begin{ass}\label{ass_sde}
	\begin{enumerate}
		\item $b\colon [0,T]\times \R^n\times \Omega\to  \R^n$ is Lipschitz continuous w.r.t. $x$ uniformly w.r.t. $(t,\omega)$ (i.e. $|b(t,x)-b(t,y)|\leq c|x-y|$) and $b(.,x)\in \hMT$ for each $x\in \R^n$.
	\item $\sigma\colon [0,T]\times \R^n\times \Omega\to  \R^{d\times n}$ is Lipschitz continuous w.r.t. $x$ uniformly w.r.t. $(t,\omega)$ (i.e. $|\sigma(t,x)-\sigma(t,y)|\leq c|x-y|$) and each row of $\sigma(.,x)$ belongs to $\hMT$ for each $x\in \R^n$.
		\item $K\colon [0,T]\times \R^n\times \R^d\times \Omega\to  \R^n$ is Lipschitz continuous w.r.t. $x$ uniformly w.r.t. $(t,\omega,z)$ (i.e. $|K(t,x,z)-K(t,y,z)|\leq c|x-y|$) and $K(.,x,.)\in \hhR$ for each $x\in \R^n$.
		\item $h\colon [0,T]\times \R^n\times \Omega\to  (\s^d)^{\times n}$ is Lipschitz continuous w.r.t. $x$ uniformly w.r.t. $(t,\omega)$ (i.e. $|h(t,x)-h(t,y)|\leq c|x-y|$) and $h(.,x)$ for each $x\in \R^n$ is a symmetric $d\times d$ matrix with each element taking values in $\hMT$.
	\end{enumerate}
\end{ass}

\begin{defin}
	The solution of the SDE (\ref{eq_sde}) with the initial condition $y_0\in\R^n$ is the process $Y\in \hMT$, satisfying
	\begin{align*}
	Y^i_t=&y_0+\int_0^t b^i(s,Y_s)ds+\int_0^t h^i(s,Y_{s})\colon d\qB_s\\&+\int_0^t \sigma^i(s,Y_{s})\cdot dB_s+\int_0^t \int_{\r0}K^i(s,Y_{s-},z)L(dz,ds).
	\end{align*}
\end{defin}

\begin{tw}\label{tw_existence_sde}
	Under the Assumption \ref{ass_sde} there exists the unique solution of the SDE (\ref{eq_sde}) with the initial condition $y_0\in \R^n$.
\end{tw}
\begin{proof}
	The proof is standard. We introduce the mapping
	\[
		\Lambda\colon \hMT\to \hMT
	\]
	by setting $\Lambda^i_t,\ t\in[0,T],\ i=1,\ldots,n$ as
	\begin{align*}
		\Lambda^i_t(Y):=&y_0+\int_0^t b^i(s,Y_s)ds+\int_0^t h^i(s,Y_{s-})\colon d\qB_s\\&+\int_0^t \sigma^i(s,Y_{s-})\cdot dB_s+\int_0^t \int_{\r0}K^i(s,Y_{s-},z)L(dz,ds).
	\end{align*}
	By the Lipschitz continuity of the coefficients we easily obtain the following estimate
	\begin{align*}
		\GE&\left[|\Lambda_t(Y)-\Lambda_t(Y')|^2\right]\leq C\GE\left[\int_0^t |b(s,Y_s)-b(s,Y'_s)|^2ds\right]\\
&+C\GE\left[\left|\int_0^t (h(s,Y_{s})-h(s,Y'_{s}))\colon d\qB_s\right|^2\right]+C\GE\left[\left|\int_0^t (\sigma(s,Y_{s})-\sigma(s,Y'_{s}))\cdot dB_s\right|^2\right]\\		
&+C\GE\left[\left|\int_0^t \int_{\r0}(K(s,Y_{s-},z)-K(s,Y'_{s-},z))L(dz,ds)\right|^2\right]\\
	&\leq		C\int_0^t \GE\left[|b(s,Y_s)-b(s,Y'_s)|^2\right]ds+C\GE\left[\int_0^t |h(s,Y_{s})-h(s,Y'_{s})|^2ds\right]
	\\&+C\int_0^t \GE\left[|\sigma(s,Y_{s})-\sigma(s,Y'_{s})|^2 \right]ds\\		
&+C\int_0^t \GE\left[\sup_{v\in\v}\,\int_{\r0}|K(s,Y_{s-},z)-K(s,Y'_{s-},z)|^2v(dz)\right]ds\\
		&\leq C \int_0^t \GE\left[|Y_s-Y'_s|^2\right]ds,\quad t\in[0,T].
	\end{align*}
	We have applied here some standard inequalities, domination of the quadratic covariation differential by $dt$, the continuity of the stochastic integrals w.r.t. the appropriate norms and the Lipschitz continuity of the coefficients. Note that he constant $C$ may vary from line to line, but depend only on the Lipschitz constant, dimensions $d$ and $n$, time horizon $T$ and the set $\u$.
	
	Now multiplying the both sides od inequality above by $e^{-2Ct}$ and integrating them on $[0,T]$, one gets
	\begin{align*}
		\int_0^T\GE\left[|\Lambda_t(Y)-\Lambda_t(Y')|^2\right]e^{-2Ct}dt&\leq C\int_0^T e^{-2Ct}\int_0^t \GE\left[|Y_s-Y'_s|^2\right]dsdt\\
		&\leq C\int_0^T\int_s^T e^{-2Ct} \GE\left[|Y_s-Y'_s|^2\right]dtds\\
		&=\frac{1}{2}\int_0^T(e^{-2Cs}-e^{-2CT})\GE\left[|Y_s-Y'_s|^2\right]ds.
	\end{align*}
	Thus we have the following inequality
	\[
		\int_0^T\GE\left[|\Lambda_t(Y)-\Lambda_t(Y')|^2\right]e^{-2Ct}dt\leq \frac{1}{2}\int_0^T\GE\left[|Y_t-Y'_t|^2e^{-2Ct}\right]dt,
	\]
	This equality shows that $\Lambda$ is a contraction mapping on $\hMT$ equipped with the norm
	\[
		\left(\int_0^T\GE\left[|Y_t|^2e^{-2Ct}\right]dt\right)^{1/2},
	\]
	which is equivalent to the norm $\|.\|_{\hMT}$. As a consequence there exists a unique fixed point of $\Lambda$, which is the solution of our SDE.
\end{proof}

\subsection{BSDE's and decoupled FBSDE's}
We will consider the following type of BSDE:
\begin{align}\label{eq_bsde}
	dY^i_t=&\GE\left[\xi^i+\int_t^T b^i(s,Y_s)ds+\int_t^T h^i(s,Y_{s})\colon d\qB_s \Big|\Omega_t\right],\quad t\in[0,T],
\end{align}
where $i=1,\ldots,n$, $Y=(Y^1,\ldots,Y^n)$ and $\xi=(\xi^1,\ldots, \xi^n)$.  Denote  $b=(b^1,\ldots,b^n)$ and \break $h=(h^1,\ldots,h^n)$.

We will work under following standard assumptions.
\begin{ass}\label{ass_bsde}
	\begin{enumerate}
		\item $\xi^i\in L^1_G(\Omega_T)$.
		\item $b\colon [0,T]\times \R^n\times \Omega\to  \R^n$ is Lipschitz continuous w.r.t. $x$ uniformly w.r.t. $(t,\omega)$ (i.e.  $|b(t,x)-b(t,y)|\leq c|x-y|$) and $b(.,x)\in \hM1T$ for each $x\in \R^n$.
		\item $h\colon [0,T]\times \R^n\times \Omega\to  (\s^d)^{\times n}$ is Lipschitz continuous w.r.t. $x$ uniformly w.r.t. $(t,\omega)$ (i.e. $|h(t,x)-h(t,y)|\leq c|x-y|$) and $h(.,x)$ for each $x\in \R^n$ is a random matrix such that each element belongs to $\hM1T$.
	\end{enumerate}
\end{ass}

\begin{defin}
	The solution of the BSDE (\ref{eq_bsde}) is the process $Y\in \hM1T$, satisfying eq.~(\ref{eq_bsde}).
\end{defin}

\begin{tw}\label{tw_existence_bsde}
	Under the Assumption \ref{ass_bsde} there exists a unique solution of the BSDE (\ref{eq_bsde}).
\end{tw}
\begin{proof}
	The proof is very similar to the SDE case and is nearly identical to the proof for the $G$-Brownian motion case (see Theorem V.2.2 in \cite{Peng_skrypt}), so we omit it.
\end{proof}

As the concequence, we can introduce decoupled FBSDE's. For simplicity assume $n=1$.

Consider the following SDE
\begin{align}\label{eg_SDE_in_FBSDE}
	dX_s^{t,\xi}&= b(X^{t,\xi}_s)ds+h(X^{t,\xi}_s) \colon d\qB_s+\sigma(X^{t,\xi}_s) \cdot dB_s+\int_{\r0}K(X^{t,\xi}_{s-},z)L(dz,ds),\notag\\
	X_t^{t,\xi}&=\xi,\quad s\in[t,T],
\end{align}
where $\xi\in L^2_G(\Omega_t)$, $\sigma\colon \R\to \R^d$, $b\colon \R\to \R$, $h\colon \R\to \s^d$ and $K\colon \R\times\R^d\to\R$ are deterministic and Lipschitz continuous functions  w.r.t. $(x,y)$ (and uniformly in $z$ for $K$). By Theorem \ref{tw_existence_sde}, there is a unique solution of this SDE. Now consider associated BSDE
\begin{align}\label{eeq_BSDE_in_FBSDE}
	Y^{t,\xi}_s=&\GE\left[\Phi(X^{t,\xi}_T)+\int_s^T f(X^{t,\xi}_r,Y^{t,\xi}_r)dr+\int_s^T g(X^{t,\xi}_r,Y^{t,\xi}_{r}) d\qB_r\Big|\Omega_s\right],\quad s\in[t,T],
\end{align}
where $\Phi\colon \R\to\R$, $f\colon\R\times\R\to\R$ and $g\colon\R\times\R\to\s^d$ are deterministic and Lipschitz continuous functions w.r.t. $(x,y)$. Theorem \ref{tw_existence_bsde} guarantees the existence and uniqueness of the solution. The pair $(X_s^{t,\xi}, Y_s^{t,\xi})$ is the decoupled FBSDE with jumps. 

In the forthcoming work we will deal with the optimal control of these systems driven by the $G$-\levy processes with finite activity.

\section*{Appendix}

In the Appendix we will prove the Theorem \ref{tw_DPP}, \ref{tw_viscosity_sol_iPDE} and \ref{tw_representation_for_lip}. We will follow exactly the argument presented by Denis \emph{et al.} in \cite{Denis_function_spaces}, so for the sake of consistency of the paper we will just sketch the reasoning. Just as in \cite{Denis_function_spaces} the proofs will be preceded by a series of lemmas investigating the properties of some essential supremums. Namely, for $\zeta\in L^2(\tilde\Omega,\F_t,\P_0;\R^n)$, $0\leq s\leq t<T$ and a fixed $\phi\in C_{b,Lip}(\R^n\times\R^d)$ we introduce
\[
	\Lambda_{t,T}[\zeta]:=\esup_{\theta\in\a^{\u}_{t,T}}\E^{\P_0}[\phi(\zeta,B^{t,\theta}_T)|\F_t].
\]

\begin{lem}\label{lem_max_property_lambda}
	For any $\theta^1,\,\theta^2\in\a^{\u}_{t,T}$ there exists $\theta\in \a^{\u}_{t,T}$ s.t.
	\begin{equation}\label{eq_max_lambda_1}
		\E^{\P_0}[\phi(\zeta,B^{t,\theta}_T)|\F_t]=\E^{\P_0}[\phi(\zeta,B^{t,\theta^1}_T)|\F_t]\vee \E^{\P_0}[\phi(\zeta,B^{t,\theta^2}_T)|\F_t].
	\end{equation}
	Consequently, there exists a sequence $\{\theta^i\}_{i=1}^{\infty}$ in $\a^{\u}_{t,T}$ s.t.
	\begin{equation}\label{eq_max_lambda_2}
		\E^{\P_0}[\phi(\zeta,B^{t,\theta^i}_T)|\F_t]\nearrow \Lambda_{t,T}[\zeta],\ \P_0-a.s.
	\end{equation}
	Moreover, for each $s\leq t$ we have
	\begin{equation}\label{eq_max_lambda_and_Yan_commut}
		\E^{\P_0}[\esup_{\theta\in\a^{\u}_{t,T}}\E^{\P_0}[\phi(\zeta,B^{t,\theta}_T)|\F_t]|\F_s]=\esup_{\theta\in\a^{\u}_{t,T}}\E^{\P_0}[\phi(\zeta,B^{t,\theta}_T)|\F_s].
	\end{equation}
\end{lem}

\begin{proof}
	The proof is identical to the proof of Lemma 41 in \cite{Denis_function_spaces}. We define a set\linebreak $A:=\{\omega\in\tilde\Omega\colon \E^{\P_0}[\phi(\zeta,B^{t,\theta^1}_T)|\F_t]\geq \E^{\P_0}[\phi(\zeta,B^{t,\theta^2}_T)|\F_t]\}$ and we put 
	\[
		\theta_s:=\I_A\theta^1_s+\I_{A^c}\theta^2_s.
	\]
	Then $\theta\in\a^{\u}_{t,T}$ and we get directly eq. \eqref{eq_max_lambda_1} and \eqref{eq_max_lambda_2}. Eq. \eqref{eq_max_lambda_and_Yan_commut} is the consequence of eq. \eqref{eq_max_lambda_1} and Yan's commutation theorem (see Theorem A3 in \cite{Peng_Yan}.)
\end{proof}

\begin{lem}\label{lem_properies_lambda}
	For every $0\leq t<T$ we have that $\Lambda_{t,T}[.]\colon L^2(\tilde\Omega,\F_t,\P_0;\R^n)\to L^2(\tilde\Omega,\F_t,\P_0;\R)$ is bounded by the bound of $\phi$ and Lipschitz continuous with the same constant as $\phi$, i.e.
	\begin{enumerate}
		\item $\Lambda_{t,T}[\zeta]\leq C_{\phi}$,
		\item $|\Lambda_{t,T}[\zeta]-\Lambda_{t,T}[\zeta']|\leq L_{\phi} |\zeta-\zeta'|$.
		\item $\Lambda_{t,T}[x]$ is deterministic for every $x\in\R^n$ and $\Lambda_{t,T}[x]=\Lambda_{0,T-t}[x]$.
		\item Put $u_{t,T}(x):=\Lambda_{t,T}[x]$ for every $x\in\R^n$. Then for every $\zeta\in L^2(\tilde\Omega,\F_t,\P_0;\R^n)$ we have
		\[
			u_{t,T}(\zeta)=\Lambda_{t,T}[\zeta],\ \P_0-a.s.
		\]
	\end{enumerate}
	\end{lem}
\begin{proof}
	Point 1 follows directly from the definition of $\Lambda$ and boundedness of $\phi$, whereas point 2 is the consequence of sublinearity of essential supremum and Lipschitz continuity of $\phi$ For details see Lemma 42 of \cite{Denis_function_spaces}. 
	
	To prove point 3 we introduce the set
	\[
		\tilde \a_{t,T}^{\u}:=\{(\theta^d,\theta^{1,c},\theta^{2,c})\in \a_{t,T}^{\u}\colon (\theta^{1,c},\theta^{2,c})\textrm { is } \fil^t \textrm{-adapted and }\theta^d \textrm { is } \fil^t \textrm{-predictable}\}.
	\]
Then we have that for any $\theta\in\tilde\a^{\u}_{t,T}$ we have that $B^{t,\theta}_T$ is independent of $\F_t$ and consequently
	\begin{align}\label{eq_lemma_lambda_eq1}
		\sup_{\theta\in\tilde\a_{t,T}^{\u}}\E^{\P_0}[\phi(x,B^{t,\theta}_T)]=\esup_{\theta\in\tilde\a_{t,T}^{\u}}\E^{\P_0}[\phi(x,B^{t,\theta}_T)|\F_t]\leq \Lambda_{t,T}[x].
	\end{align}
		Note also that
	\[
		\tilde \a:=\left\{\sum_{j=1}^N\I_{A_j}\theta^j\colon \{A_j\}_{j=1}^N \textrm{ is an }\F_t \textrm{-partition of }\tilde\Omega,\ \theta^j\in \tilde\a^{\u}_{t,T} \right\}
	\]
	is dense in $\a^{\u}_{t,T}$. Hence
	\begin{align}\label{eq_lemma_lambda_eq2}
	\Lambda_{t,T}[x]&=\esup_{\sum_{j=1}^N\I_{A_j}\theta^j\in\tilde\a}\E^{\P_0}[\phi(x,B^{t,\sum_{j=1}^N\I_{A_j}\theta^j}_T)|\F_t] =\esup_{\sum_{j=1}^N\I_{A_j}\theta^j\in\tilde\a}\sum_{j=1}^N\I_{A_j}\E^{\P_0}[\phi(x,B^{t,\theta^j}_T)]\notag\\
	&\leq \esup_{\sum_{j=1}^N\I_{A_j}\theta^j\in\tilde\a}\sum_{j=1}^N\I_{A_j}\sup_{\theta\in\tilde\a_{t,T}^{\u}} \E^{\P_0}[\phi(x,B^{t,\theta}_T)]=\sup_{\theta\in\tilde\a_{t,T}^{\u}} \E^{\P_0}[\phi(x,B^{t,\theta}_T)]
	\end{align}
	Compare eq. \eqref{eq_lemma_lambda_eq1} and \eqref{eq_lemma_lambda_eq2} to get point 3. (See also Lemma 43 in \cite{Denis_function_spaces}).
	
	Point 4 is easy to get for simple random variables of the form $\zeta=\sum_{i=1}^n\I_{A_i}x_i$, $x_i\in\R^n,\ A_i\in\F_t$, the assertion for general $\zeta$ is obtained via the regularity of $\Lambda_{t,T}[.]$ from point 2. See Lemma 44 in \cite{Denis_function_spaces} for details.
\end{proof}

\begin{prop}\label{prop_Dyn_consist}
	Let $\xi\in L^2(\tilde\Omega,\F_s,\P_0;\R^n)$, $\psi\in C_{b,Lip}(\R^n\times\R^d\times\R^d)$ and $0\leq s<t<T$. Then
	\[
		\esup_{\theta\in\a^{\u}_{s,T}}\E^{\P_0}[\psi(\xi,B^{s,\theta}_t,B^{t,\theta}_T)|\F_s]=
		 \esup_{\theta\in\a^{\u}_{s,t}}\,\E^{\P_0}\left[\esup_{\tilde{\theta}\in\a^{\u}_{t,T}}\E^{\P_0}[\psi(x,y,B^{t,\tilde{\theta}}_T)|\F_t]|_{\substack{x=\xi\\y=B^{s,\theta}_t}}\,\big|\F_s\right].
	\]
\end{prop}

\begin{proof}
This is the consequence of eq. \eqref{eq_max_lambda_and_Yan_commut} and Lemma \ref{lem_properies_lambda}, point 4.
\end{proof}
Now we are ready to prove DPP.
\begin{proof}[Proof of Theorem \ref{tw_DPP}]
	The assertion is just an easy application of Proposition \ref{prop_Dyn_consist}. To be more specific, fix $h\in[0,T-t]$. By the definition of $u$ we have
	\begin{align*}
		u(t,x)&=\sup_{\theta\in\a_{t,T}^{\u}}\E^{\P_0} [\phi(x+B^{t,\theta}_T)]=\sup_{\theta\in\a_{t,T}^{\u}}\E^{\P_0} \left[\phi\left(x+B^{t,\theta}_{t+h}+B^{t+h,\theta}_T\right)|\F_t\right].
	\end{align*}
	Hence applying Proposition \ref{prop_Dyn_consist} to $\psi(x,y,z)=\phi(x+y+z)$ we get that
	\begin{align*}
		u(t,x)&=\sup_{\theta\in\a^{\u}_{t,t+h}}\E^{\P_0}\left[\sup_{\tilde \theta\in\a^{\u}_{t+h,T}} \E^{\P_0}\left[\phi(x+y+B^{t+h,\tilde \theta}_T)|\F_{t+h}\right]|_{y=B^{t,\theta}_{t+h}} |\F_t \right]\\&=\sup_{\theta\in\a^{\u}_{t,t+h}}\E^{\P_0}\left[u(t+h,x+B^{t,\theta}_{t+h})\right]\qedhere
	\end{align*}
\end{proof}

\begin{proof}[Proof of Theorem \ref{tw_viscosity_sol_iPDE}]
Let $\psi\in C_b^{2,3}([0,T]\times\R^d)$ be such that $\psi\geq u$ and for a fixed\break $(t,x)\in[0,T]\times \R^d$ we have $\psi(t,x)=u(t,x)$.

We introduce now the integral with jumps cut at the fixed level $\epsilon>0$:
\[
	B_T^{t,\theta,\epsilon}:=\int_t^T \theta^{1,c}_sds+\int_t^T\theta_s^{2,c}dW_s+\int_t^T\int_{\R^{d}\setminus B(0,\epsilon)}\theta^d(s,z)N(ds,dz).
\]
By the \ito formula we have
\begin{align*}
	\psi&\left(t+h,x+B^{t,\theta,\epsilon}_{t+h}\right)-\psi(t,x)=\int_t^{t+h}\frac{\partial\psi}{\partial s}\left(s,x+B^{t,\theta,\epsilon}_{s-}\right)ds\\&+\int_t^{t+h}\langle D\psi\left(s,x+B^{t,\theta,\epsilon}_{s-}\right),\theta^{1,c}_s\rangle ds+\int_t^{t+h}\langle D\psi\left(s,x+B^{t,\theta,\epsilon}_{s-}\right),\theta^{2,c}_sdW_s\rangle\\
	 &+\int_t^{t+h}\frac{1}{2}\tr\left[\theta^{2,c}_s\left(\theta^{2,c}_s\right)^TD^2\psi\left(s,x+B^{t,\theta,\epsilon}_{s-}\right)\right]ds\\
	 &+\int_t^{t+h}\int_{\R^{d}\setminus B(0,\epsilon)}\left[\psi\left(s,x+B^{t,\theta,\epsilon}_{s-}+\theta^d(s,z)\right)-\psi\left(s,x+B^{t,\theta,\epsilon}_{s-}\right)\right]N(ds,dz).
\end{align*}
Taking $\E^{\P_0}$ we get the following
\[
	 \E^{\P_0}\left[\psi\left(t+h,x+B^{t,\theta,\epsilon}_{t+h}\right)-\psi(t,x)\right]=:\E^{\P_0}[I_1]+\E^{\P_0}[I_2]+\E^{\P_0}[I_3]+\E^{\P_0}[I_4]+\E^{\P_0}[I_5].
\]
Note that $\E^{\P_0}[I_3]=0$.
Moreover,  $\left\{\frac{\partial\psi}{\partial s}+\langle D\psi,\theta^{1,c}_s\rangle+\frac{1}{2}\tr\left[\theta^{2,c}_s\left(\theta^{2,c}_s\right)^TD^2\psi\right]\right\}(s,y)$ is uniformly Lipschitz continuous in $(s,y)$. Note also the following estimate
\[
	\sup_{\theta\in\a^{\u}_{t,t+h}}\E^{\P_0}[|B_{t+h}^{t,\theta,\epsilon}|]\leq C[h^{1/2}+h]
\]
where the constant $C$ is independent of $\epsilon$ and $h$.

Combining those two results we get the following estimate (the constant $C$ may vary from line to line)
\begin{align}\label{eq_viscosity_estim1}
	\E^{\P_0}&\left[I_1+I_2+I_4\right] \notag\\&=\E^{\P_0}\left[\int_t^{t+h}\left\{\frac{\partial\psi}{\partial s}+\langle D\psi,\theta^{1,c}_s\rangle +\frac{1}{2}\tr\left[\theta^{2,c}_s\left(\theta^{2,c}_s\right)^TD^2\psi\right]\right\}\left(s,x+B^{t,\theta,\epsilon}_{s-}\right)ds\right] \notag\\
	&\leq \E^{\P_0}\left[\int_t^{t+h}\left\{\frac{\partial\psi}{\partial s}+\langle D\psi,\theta^{1,c}_s\rangle +\frac{1}{2}\tr\left[\theta^{2,c}_s\left(\theta^{2,c}_s\right)^TD^2\psi\right]\right\}\left(t,x\right)ds\right] \notag\\&\quad+C(h^{1/2}+h).
\end{align}

Similarly
\begin{align}\label{eq_viscosity_estim2}
	 \E^{\P_0}\left[I_5\right]&=\E^{\P_0}\left[\int_t^{t+h}\!\!\int_{|z|\geq \epsilon}\left[\psi\left(s,x+B^{t,\theta,\epsilon}_{s-}+\theta^d(s,z)\right)-\psi\left(s,x+B^{t,\theta,\epsilon}_{s-}\right)\right]N(ds,dz)\right] \notag\\
	 &=\E^{\P_0}\left[\int_t^{t+h}\!\!\int_{|z|\geq \epsilon}\left[\psi\left(s,x+B^{t,\theta,\epsilon}_{s-}+\theta^d(s,z)\right)-\psi\left(s,x+B^{t,\theta,\epsilon}_{s-}\right)\right] \mu(dz)ds\right] \notag\\
	&\leq \E^{\P_0}\left[\int_t^{t+h}\int_{|z|\geq \epsilon}\!\!\left[\psi\left(s,x+\theta^d(s,z)\right)-\psi(s,x)\right]\mu(dz) ds\right]\notag\\&\quad+\mu(\{|z|\geq \epsilon\})\,C(h+h^{1/2})h 
\end{align}

Finally, let us look for an estimate for $\E^{\P}[\psi(t+h,x+B^{t,\theta}_{t+h})-\psi(t+h,x+B^{t,\theta,\epsilon}_{t+h})]$. We define the random measure $\pi^{\theta}_s(\omega)=v\in\v$ iff $\theta^d(s,.)(\omega)=g_v$ and $\R^{\epsilon,\theta}_s(\omega):=\theta^d(s,\R^d\setminus B(0,\epsilon))(\omega)$. Note that the function $(s,\omega)\mapsto \pi^{\theta}_s(A)(\omega)$ ($A\in\B(\R^d)$ is fixed) is $\B([0,T])\otimes \F_T$-measurable. Then by Lipschitz continuity of $\psi$ we get for some $0<q<1$ that
\begin{align}\label{eq_viscosity_estim3}
	\E^{\P_0}&\left[\psi(t+h,x+B^{t,\theta}_{t+h})-\psi(t+h,x+B^{t,\theta,\epsilon}_{t+h})\right]\leq C \E^{\P_0}\left[\left|\int_t^{t+h}\!\!\int_{0<|z|<\epsilon}\theta^d(s,z)N(ds,dz)\right|\right] \notag\\
	&\leq C\, \E^{\P_0}\left[\int_t^{t+h}\int_{0<|z|<\epsilon}\left|\theta^d(s,z)\right| \mu(dz)ds\right]
	\leq	C\epsilon^{1-q}\, \E^{\P_0}\left[\int_t^{t+h}\int_{\r0}\left|\theta^d(s,z)\right|^q \mu(dz)ds\right]\notag\\
	&=C\epsilon^{1-q}\, \E^{\P_0}\left[\int_t^{t+h}\!\!\int_{\r0}|z|^q\ \pi^{\theta}_s(dz)ds\right]\leq C\epsilon^{1-q}\, h\, \sup_{v\in\v}\,\int_{\r0}|z|^q\ v(dz).
\end{align}

Connecting the estimates \eqref{eq_viscosity_estim1}, \eqref{eq_viscosity_estim2} and \eqref{eq_viscosity_estim3} with Theorem \ref{tw_DPP} we get
\begin{align*}
	0&=\sup_{\theta\in\a^{\u}_{t,t+h}}\!\!\E^{\P_0}\left[u(t+h,x+B^{t,\theta}_{t+h})-u(t,x)\right]	
	\leq\sup_{\theta\in\a^{\u}_{t,t+h}}\!\!\E^{\P_0}\left[\psi(t+h,x+B^{t,\theta}_{t+h})-\psi(t,x)\right]\displaybreak[1]\\
	&=\sup_{\theta\in\a^{\u}_{t,t+h}}\!\!\E^{\P_0}\left[\psi(t+h,x+B^{t,\theta}_{t+h})-\psi(t+h,x+B^{t,\theta,\epsilon}_{t+h})+\psi(t+h,x+B^{t,\theta,\epsilon}_{t+h})-\psi(t,x)\right]
		\\
	&\leq\sup_{\theta\in\a^{\u}_{t,t+h}}\left\{\E^{\P_0}\left[\int_t^{t+h}\left\{\frac{\partial\psi}{\partial s}+\langle D\psi,\theta^{1,c}_s\rangle +\frac{1}{2}\tr\left[\theta^{2,c}_s\left(\theta^{2,c}_s\right)^TD^2\psi\right]\right\}\left(t,x\right)ds\right.\right.\\
	&\quad+\left.\left.\int_t^{t+h}\int_{|z|\geq \epsilon}\!\!\left[\psi\left(s,x+\theta^d(s,z)\right)-\psi(s,x)\right]\mu(dz)ds\right]\right\}\\
	&\quad+\mu(\R^d\setminus B(0,\epsilon))\,C(h+h^{1/2})h+\sup_{v\in\v}\,\int_{\r0}|z|^q\ v(dz)\,C\,h\epsilon^{1-q}.
\end{align*}
We divide both sides by $h$ and go with it to $0$ ($\epsilon$ is still fixed). We get then by Lebesgue differentiation theorem that:
\begin{align*}
	0
	&\leq\sup_{\theta\in\a^{\u}_{t,t+h}}\left\{\E^{\P_0}\left[\left\{\frac{\partial\psi}{\partial t}+\langle D\psi,\theta^{1,c}_t\rangle +\frac{1}{2}\tr\left[\theta^{2,c}_t\left(\theta^{2,c}_t\right)^TD^2\psi\right]\right\}\left(t,x\right)\right.\right.\\
	&\quad+\left.\left.\int_{|z|\geq \epsilon}\!\!\left[\psi\left(t,x+\theta^d(t,z)\right)-\psi(t,x)\right]\mu(dz)\right]\right\}
	+\sup_{v\in\v}\,\int_{\r0}|z|^q\ v(dz)\,C\epsilon^{1-q}.
\end{align*}
Now we go with $\epsilon$ to $0$. Since $\sup_{v\in\v}\int_{B(0,1)}|z|^qv(dz)<\infty$ for some $0<q<1$ (see Assumption \ref{ass1}) and $\sup_{v\in\v}\int_{|z|\geq1}|z|^qv(dz)\leq \sup_{v\in\v}\int_{|z|\geq1}|z|v(dz)<\infty$, we conclude that the last term will decrease to $0$. We also note that $\{|z|>\epsilon\}$ will increase to $\r0$. Therefore we can write
\begin{align*}
	0
	&\leq\sup_{\theta\in\a^{\u}_{t,t+h}}\left\{\E^{\P_0}\left[\left\{\frac{\partial\psi}{\partial t}+\langle D\psi,\theta^{1,c}_t\rangle +\frac{1}{2}\tr\left[\theta^{2,c}_t\left(\theta^{2,c}_t\right)^TD^2\psi\right]\right\}\left(t,x\right)\right.\right.\\
	&\quad+\left.\left.\int_{\r0}\left[\psi\left(t,x+\theta^d(t,z)\right)-\psi(t,x)\right] \mu(dz)\right]\right\}
	\\
	&=\sup_{\theta\in\a^{\u}_{t,t+h}}\left\{\E^{\P_0}\left[\left\{\frac{\partial\psi}{\partial t}+\langle D\psi,\theta^{1,c}_t\rangle +\frac{1}{2}\tr\left[\theta^{2,c}_t\left(\theta^{2,c}_t\right)^TD^2\psi\right]\right\}\left(t,x\right)\right.\right.\\
	&\quad+\left.\left.\int_{\r0}\left[\psi\left(t,x+z\right)-\psi(t,x)\right] \pi^{\theta}_t(dz)\right]\right\}
	\\
	&=\!\sup_{(v,p,Q)\in\u}\!\left\{\left[\frac{\partial\psi}{\partial t}+\langle D\psi,p\rangle +\frac{\tr\left[QQ^TD^2\psi\right]}{2}\right](t,x)
	+\int_{\r0}\left[\psi(t,x+z)-\psi(t,x)\right] v(dz)\right\}.
\end{align*}
The first inequality is the consequence of changing the measure formula for pushforward measures. The random measure $\pi^{\theta}_t$ is defined as by equation \eqref{eq_viscosity_estim3} and takes values in $\v$ for almost all $(t,\omega)$. The last equality is the consequence of the fact that deterministic integrands belong to $\a^{\u}_{0,T}$. We can see now that $u$ is the viscosity subsolution of the integropartial PDE in Theorem \ref{tw_viscosity_sol_iPDE}. By the same argument one can prove that $u$ is also a supersolution.
\end{proof}

\begin{proof}[Proof of Theorem \ref{tw_representation_for_lip}]
The assertion of this theorem is just an easy consequence of Proposition \ref{prop_Dyn_consist}, the definition of the sublinear expectation $\GE[.]$, Theorem \ref{tw_viscosity_sol_iPDE} and the uniqueness of the  viscosity solution of the IPDE. 
\end{proof}

\section*{Acknowledgements}
The author would like to thank An Ta Thi Kieu (Centre of Mathematics for Applications, University of Oslo) for many fruitful discussions and careful reading of the manuscript. Moreover, the author is very grateful to the anonymous referee for helpful remarks to Section 3.

\end{document}